\documentclass[dvipsnames,svgnames,12pt]{article}

\usepackage{microtype}
\usepackage{graphicx}
\usepackage{subfigure}
\usepackage{booktabs} 
\usepackage{hyperref}
\hypersetup{
    colorlinks=true,
    linkcolor=blue,
    filecolor=magenta,      
    urlcolor=cyan,
    citecolor=Emerald
    }

\usepackage[margin=3cm]{geometry}

\usepackage{amsmath}
\usepackage{amssymb}
\usepackage{mathtools}
\usepackage{amsthm}

\usepackage[textwidth=3cm]{todonotes}

\usepackage[capitalize,noabbrev]{cleveref}

\theoremstyle{plain}
\newtheorem{theorem}{Theorem}[section]
\newtheorem{propo}[theorem]{Proposition}
\newtheorem{lemma}[theorem]{Lemma}

\newtheorem{assumption}{Assumption}
\theoremstyle{definition}

\crefname{assumption}{assumption}{assumptions}

\theoremstyle{remark}

\newtheorem{example}[theorem]{Example}

\usepackage{dsfont} 

\global\long\def\esp{\mathbb{E}}%
\global\long\def\F{\mathcal{F}}%
\global\long\def\R{\mathbb{R}}%
\global\long\def\P{\mathbb{P}}%

\newcommand{\ind}{\mathds{1}}

\newcommand{\imp}{\rm{imp}}

\usepackage[textwidth=3cm]{todonotes}

\definecolor{OliveGreen}{rgb}{0.24, 0.71, 0.54}
\definecolor{RoyalBlue}{rgb}{0.0, 0.47, 0.75}
\definecolor{BrickRed}{rgb}{0.77, 0.12, 0.23}
\definecolor{Vert}{RGB}{0,128,0}

\usepackage[normalem]{ulem}

\usepackage[round]{natbib}

\usepackage{authblk}
\title{Random features models: a way to study the success of naive imputation}
\author[1]{Alexis Ayme}
\author[1,2]{Claire Boyer}
\author[3]{Aymeric Dieuleveut}
\author[1]{Erwan Scornet}
\affil[1]{Sorbonne Université and Université Paris Cité, CNRS, Laboratoire de Probabilités, Statistique et Modélisation, F-75005 Paris, France}
\affil[2]{Institut Universitaire de France (IUF)}
\affil[3]{CMAP, Ecole Polytechnique}
\date{}    

\begin{document}

  \maketitle

\begin{abstract}
Constant (naive) imputation is still widely used in practice as this is a first easy-to-use technique to deal with missing data. 
Yet, this simple method could be expected to induce a large bias for prediction purposes, as the imputed input may strongly differ from the true underlying data. 
However, recent works suggest that this bias is low in the context of high-dimensional linear predictors when data is supposed to be missing completely at random (MCAR). 
This paper completes the picture for linear predictors by confirming the intuition that the bias is negligible  and  that surprisingly naive imputation also remains relevant in very low dimension.
To this aim, we consider a unique underlying random features model, which offers a rigorous framework for studying predictive performances, whilst the dimension of the observed features varies.
Building on these theoretical results, we establish finite-sample bounds on stochastic gradient (SGD) predictors applied to zero-imputed data, a strategy particularly well suited for large-scale learning.
If the MCAR assumption appears to be strong, we show that similar favorable behaviors occur for more complex missing data scenarios.
\end{abstract}

\section{Introduction}

Missing data appear in most real-world datasets as they arise from merging different data sources, data collecting issues, self-censorship in surveys, just to name a few. Specific handling techniques are required, as most machine learning algorithms do not natively handle missing values. A common practice consists in imputing missing entries. The resulting complete dataset can then be analyzed using any machine learning algorithm. 

While there exists a variety of imputation strategies \citep[single, multiple, conditional, marginal imputation~; see, e.g.,][for an overview]{bertsimas2018predictive}, mean imputation is definitely one of the most common practices. Such a procedure has been largely criticized in the past as (single) mean imputation distorts data distributions by lowering variances, which can lead to inconsistent parameter estimation. Indeed, a large part of the literature on missing values focuses on inference in parametric models, such as linear \citep[][]{little1992regression, jones1996indicator} or logistic models \citep[][]{consentino2011missing, jiang2020logistic}. 
From an empirical perspective, benchmarks of imputation techniques \citep{woznica2020does} indicate that simple imputation, such as the mean, induces reasonable predictive performances, compared to more complex imputation techniques such as MICE \citep[][]{perez2022benchmarking}.  

On the contrary, a recent line of work \citep[][]{josse2019consistency} aims at studying the predictive performances of impute-then-regress strategies that work by first imputing data (possibly with a very simple procedure) and then fitting a learning algorithm on the imputed dataset. Whereas mean imputation leads to inconsistent \textit{parameter estimation}, \citet{josse2019consistency} and \citet{bertsimas2021beyond} show that impute-then-regress procedures happen to be consistent if the learning algorithm is universally consistent. 
\citet{le2021sa} generalize the consistency results of mean-imputation by \citet{josse2019consistency, bertsimas2021beyond} and prove that for any universally consistent regression model, almost all single imputation strategies can lead to consistent predictors. Therefore, the impact of a specific imputation strategy has to be analyzed for specific regression models.

Without dispute, linear models are the most classic regression framework. However, their study becomes challenging in presence of missing values as they can require to build $2^d$ non-linear regression models (one for each missing data pattern), where $d$ is the number of input variables \citep[][]{le2020linear,ayme2022near}. In the context of linear models with missing inputs, \citet{le2020linear} establish finite-sample generalization bounds for zero-imputation, showing that this strategy is generally inconsistent. However, assuming a low-rank structure on the input variables, \citet{ayme2023naive} prove that zero-imputation prior to learning is consistent in a high-dimensional setting. Note that the impact of zero-imputation with low-rank inputs has also been analyzed by \citet{agarwal2019robustness} in the context of Principal Components Regression, where the same type of generalization bounds were established. In this paper, we want to go beyond the low-rank structure by considering a (possibly infinite) latent space, and using the random feature framework.

\paragraph{Related work - Random features}
First introduced by \citet{rahimi2007random}, random features are used in nonparametric regression to approximate, with a few features, a kernel method where the final predictor belongs to an infinite-dimensional RKHS. \citet{rudi2017generalization,carratino2018learning} obtain generalization upper bounds for kernel regression learned with a small number of features, leading to computational efficiency. Random features are also used to describe a one-hidden-layer neural network \citep{bach2017breaking}. 

\paragraph{Related work - high-dimensional linear models}
Linear models have been widely studied in a fixed design, considering the input variables are fixed \citep[see, e.g., ][for an analysis in the high-dimensional case]{hastie2015statistical}. Quite notably, few works analyze (high-dimensional) linear models in the random design setting, a necessary framework to assess the predictive performance of linear models on unseen data \citep{caponnetto2007optimal,hsu2012random,mourtada2022elementary}.
These works mainly focus on the statistical properties of the Empirical Risk Minimizer (ERM) with a ridge regularization using uniform concentration bounds.
On the other hand,
\cite{bach2013non,raskutti2014early,dieuleveut2017harder} 
directly control the  generalization error of  predictors resulting of stochastic gradient strategies, while performing a single pass on the dataset. The obtained bounds have therefore the advantage of being dependent on the training algorithm involved. 


        

\paragraph{Contributions} In this paper, we analyze the impact of the classic imputation-by-zero procedure on predictive performances, as a function of the input dimension. To this aim, we consider a latent space from which an arbitrary number of input variables are built. The output variable is assumed to result from a linear transformation of the latent variable. Such a framework allows us to analyze how predictive performances vary with the number of input variables, inside a common fixed model. Under this setting, we assume that all entries of input variables are observed with probability $\rho \in (0,1)$, within a MCAR scenario, and study the performance of a linear model trained on imputed-by-zero data. 
\begin{itemize}
    \item We prove that when the input dimension $d$ is negligible compared to that of the latent space $p$, the Bayes risk of the zero-imputation strategy is negligible compared to that induced by missing data themselves. Therefore, naive imputation is on par with best strategies. 
    \item When $d \gg p$, both above errors vanish, which highlights that neither the presence of missing data or the naive imputation procedure hinders the predictive performances.
    \item From a learning perspective, we use Stochastic Gradient Descent to learn parameters on zero-imputed data. We provide finite-sample generalization bounds in different regimes, highlighting that the excess risk vanishes at $1/\sqrt{n}$ for very low dimensions ($d \ll p$) and high dimensions ($d > (1 - \rho) \sqrt{n}/ \rho$). 
    \item Two different regimes arise from the finite dimension of the latent space. To move beyond this disjunctive scenario, we consider a latent space of infinite dimension and analyze predictors built on $d$ zero-imputed input variables. We prove that the corresponding Bayes excess risk is controlled via the excess risk of a kernel ridge procedure, with a penalization constant depending on $\rho$ and $d$. A finite-sample generalization bound on the SGD strategy applied on zero-imputed data is established and shows that zero-imputation is consistent in high-dimensional regimes ($d \gg \sqrt{n}$).  
    \item The MCAR assumption considered throughout the paper, and often in the literature, can be attenuated at the cost of weaker theoretical results but which allows to show that naive imputation is relevant in high dimension even for non-trivial Missing Not At Random (MNAR) scenarios.
\end{itemize}

\textbf{Notations. } For $n\in \mathbb N$, we denote $[n] = \lbrace 1, \dots, n\rbrace$. We use $\lesssim$ to denote inequality up to a universal constant.



\section{Imputation is adapted for very low and high-dimensional data}
\label{sec:DD}

\subsection{Setting}

We adopt the classical regression framework in which we want to predict the value of an output random variable  $Y\in\R$ given an input random variable $X\in \mathcal{X} = \mathbb{R}^d$ of dimension $d$. 
More precisely, our goal is to build a predictor $\hat{f}: \mathcal{X} \to \R$ that accurately estimates the regression function $f^\star$  (also called Bayes predictor) defined as a minimizer of the quadratic risk 
\begin{align}
R(f):= \esp\left[\left(Y-f\left(X\right)\right)^2\right],
\end{align}
over the class of measurable functions $f: \mathcal{X} \to \R$. When $f$ is linear, we simply denote $R(\theta)$ the risk of the linear function parameterized by $\theta$, i.e., such that for all $x\in\mathbb{R}^d$, $f(x)=x^\top \theta$.

\paragraph{Random features} 
Real datasets are often characterized by high correlations between variables, or equivalently by a hidden low-rank structure \citep{johnstone2001,udell2019big}. 
The random feature framework \citep{rahimi2007random} constitutes a general and flexible approach for modeling such datasets.
We, therefore, assume that the inputs $(X_i)_{i\in [n]}$, i.i.d.\ copies of $X$, actually result from a random feature (RF) model. 
For pedagogical purposes, we start by restricting ourselves to finite-dimensional latent models.  
\begin{assumption}[Gaussian random features] \label{eq:modelRFgaussien} The input variables $(X_i)_{i\in [n]}$ are assumed to be given by
\begin{equation}
    X_{i,j}=  Z_i^\top  W_j , \qquad \text{for } i\in [n] \text{ and } j\in[d]
\end{equation}
where the $p$-dimensional latent variables $Z_1, \hdots, Z_n$ are i.i.d. copies of $Z \sim\mathcal{N}(0,I_p)$, and where the $p$-dimensional random weights $W_1, \hdots, W_d$ are i.i.d. copies of $W$ uniformly distributed on the sphere $\mathbb{S}^{p-1}$, i.e., $W\sim \mathcal{U}(\mathbb{S}^{p-1})$. 
\end{assumption}

The latent space in Assumption \ref{eq:modelRFgaussien} corresponds to $\R^p$. We have only access to $n$ observations $(X_i)_{i\in [n]}$ of dimension $d$ that can be seen as random projections using $d$ directions of the latent features $(Z_i)_{i\in [n]}$. The total amount of information contained in the latent variables $(Z_i)_{i\in [n]}$ cannot be recovered in expectation by that contained in the observations $(X_i)_{i\in [n]}$ if $d<p$. This no longer holds when $d\gg p$, and the observations $(X_i)_{i\in [n]}$ can be therefore regarded as low-rank variables of rank $p$.

\paragraph{Linear models with RF} 
In the following, we present the model governing the distribution of the output variable $Y$.
\begin{assumption}[Latent linear well specified model]\label{eq:modelYgaussien}
The target variable $Y$ is assumed to follow a linear model w.r.t.\ the latent covariate $Z$, i.e.,
    \begin{equation}
    Y=  Z ^\top \beta^\star + \epsilon,
\end{equation}
where the model parameter is denoted by $\beta^\star\in\R^p$ and the noise $\epsilon\sim \mathcal{N}(0,\sigma^2)$ is assumed to be independent of $Z$.
\end{assumption}
 
 Considering such a random feature setting is particularly convenient when studying the influence of the input dimension $d$ on learning without modifying the underlying model (indeed, the distribution of $Y$ --and $Z$-- does not depend on the input dimension $d$). Note that a similar model (with fixed weight $(W_j)_j$) has already been introduced in \citet{hastie2022surprises}.

\paragraph{Missing data} Often one does not have access to the full input vector $X$ but rather to a version of $X$ containing missing entries. On the contrary, the output $Y$  is always assumed to be observed. To encode such missing information on $X$,  we introduce, the random variable $P\in \{0,1\}^d$, referred to as the missing pattern (or actually an observation indicator), such that $P_{j}=1$ if  $X_{j}$ is observed and $P_{j}=0$ otherwise. Assuming that all variables are equally likely to be missing, we define $\rho:=\P(P_j=1)$ for any $j\in[d]$, i.e., $1-\rho$ is the expected proportion of missing values for any feature.
%
In this paper, we thoroughly analyze the classical Missing Completely At Random (MCAR) setting, where the missing pattern $P$ and the complete observation $(X,Y)$ are independent. Note that we will extend some of our theoretical findings for the MCAR case to relaxed missing scenarios in \Cref{sec:MNAR}.


\begin{assumption}[MCAR pattern with independent components]\label{ass:bernoulli_model}

The complete observation $(X,Y)$ and the missing pattern $P$ are assumed to be independent, i.e., $(X,Y) \perp\!\!\!\perp P$ and such that $P$ follows a Bernoulli distribution $\mathcal{B}(\rho)^{\otimes d}$, i.e., for any $j\in[d]$, $\rho=\P(P_j=1)$, with $0< \rho \leq 1$ denoting
the expected proportion of observed values.
\end{assumption}

\paragraph{Imputation} Most machine learning algorithms are not designed to deal directly with missing data. Therefore, we choose to impute the missing values (both in the training and test sets) by zero (or by the mean for non centered inputs). The imputed inputs in the train and test sets are thus denoted, for all $i$,  by  
\begin{align}
\tilde{X}_{i}= P_{i}\odot X_{i},
\end{align}
where $\odot$ represented the component-wise product. The impact of missing data, and their handling by naive imputation, in this supervised learning task can be scrutinized through the evolution of the following key quantities:
\begin{itemize}
 \item The Bayes risk based on complete random features (without missing entries): 
    \begin{equation*}
        R^\star(d):= \inf_{f} \esp \left[(Y-f(X)^2|W_1,\dots,W_d\right],
    \end{equation*}
    where the infimum is taken over all measurable functions.
    \item The Bayes risk given $X_{\rm miss}= (\tilde X,P)\in \R^d\times\{0,1\}^d$: 
    \begin{equation*}
        R_{\rm miss}^\star(d):= \inf_{f} \esp \left[(Y-f(X_{\rm miss}))^2|W_1,\dots,W_d\right],
    \end{equation*}
    where the infimum is taken over all measurable functions. It has been shown to be attained for a pattern-by-pattern predictor \citep{le2020linear}. The bias or deterministic error due to learning with missing inputs can be characterized as
    \begin{equation*}
    \Delta_{\rm miss}(d):=\esp \left[ R_{\rm miss}^\star(d)-R^\star(d)\right].
    \end{equation*}
    \item The risk of the best linear predictor relying on zero-imputed inputs: 
    \begin{equation*}
        R_{\rm imp}^\star(d):= \inf_{\theta\in\R^d} \esp \left[(Y-\tilde X^\top \theta )^2|W_1,\dots,W_d\right].
    \end{equation*}
    The approximation error associated with this specific class of predictors, among those handling missing inputs,  is denoted by 
    \begin{equation*}
    \Delta_{\rm imp/miss}(d):=\esp \left[ R_{\rm imp}^\star(d)-R_{\rm miss}^\star(d)\right].
\end{equation*}
   
\end{itemize}

Note that these three risks decrease with the dimension and are ordered as follows, 
\begin{equation}\label{eq:riskordre}
    R^\star(d)\leq R_{\rm miss}^\star(d)\leq R_{\rm imp}^\star(d).
\end{equation}
In what follows, we give a precise evaluation of $R^\star(d)$ and $R_{\rm mis}^\star(d)$ and provide bounds for $R_{\rm imp}^\star(d)$ as well. 

\subsection{Theoretical analysis}

Our goal is to dissect the systematic errors introduced by either the occurrence of missing inputs or their handling via naive zero imputation. To do so, we start by characterizing the optimal risk over the class of linear predictors when working with complete inputs.
\begin{propo}\label{prop:RiskcompletDD}
    Under \Cref{eq:modelRFgaussien,eq:modelYgaussien}, the Bayes risk for linear predictors based on complete random features is
    \[ 
        \esp \left[ R^\star(d) \right]=\left\{
        \begin{array}{ll}
        \sigma^2+ \frac{p-d}{p}\Vert \beta^\star \Vert_2^2, & \text{when } d< p,\\
        \sigma^2 & \text{when } d \geq p,
        \end{array}
        \right.
    \]
    where the expectation is taken over $(W_j)_{j\in[d]}$.
\end{propo}
 Proposition \ref{prop:RiskcompletDD} highlights that learning with a number $d$ of random features larger than the latent dimension $p$ is equivalent to learning directly with the latent covariate $Z$. Besides, when $d<p$, the Bayes predictor suffers from an increased risk, as learning is flawed by a lack of information in the (fully observed) inputs. This can be compensated by increasing the number $d$ of random features, as 
the explained variance of $Y$, i.e., 
$
    \esp Y^2-\esp R^\star(d)=\frac{d}{p}\Vert \beta^\star \Vert_2^2,
$
increases with $d$ for $d\leq p$. 

\begin{propo}\label{prop:Riskmissing}
    Under \Cref{eq:modelRFgaussien,eq:modelYgaussien,ass:bernoulli_model}, the Bayes risk for  predictors working with missing data is given by 
    \[ 
        \esp \left[ R^\star_{\rm miss }(d) \right] =\left\{
        \begin{array}{ll}
        \sigma^2+ \frac{p-\rho d}{p}\Vert \beta^\star \Vert_2^2 & \text{when } d< p,\\
         \sigma^2 + \frac{\esp[(p-B)\ind_{B\leq p}]}{p}\Vert \beta^\star \Vert_2^2 & \text{when } d \geq p,
        \end{array}
        \right.
    \]
    where the expectation is taken over the random weights $(W_j)_{j\in[d]}$ and $B\sim\mathcal{B}(d,\rho)$. Therefore,
        \begin{equation*}
        \Delta_{\rm miss}(d)=
        \left\{
        \begin{array}{ll}
         (1-\rho)\frac{d}{p}\Vert \beta^\star \Vert_2^2 & \text{when } d< p,\\
        \frac{\esp[(p-B)\ind_{B\leq p}]}{p}\Vert \beta^\star \Vert_2^2, & \text{when } d \geq p.
        \end{array}
        \right.
    \end{equation*}
\end{propo}
To our knowledge, this result is the first one to precisely evaluate the error induced by missing inputs when learning a linear latent model. 
More specifically, two regimes are identified. 
In the first regime $d<p$, i.e., when working with random features of lower dimension than that of the latent model, $R_{\rm miss}^\star$ takes the same form as $R^\star$, where the input dimension $d$ is replaced by $\rho d$.
This can be interpreted as the cost of learning with $\rho d$ observed features in expectation instead of the $d$ initial features.

In the second regime, when $d \geq p$, the error due to missing data becomes more and more negligible as $d$ increases, as the redundancy of the random feature model is sufficient to retrieve the information contained in the latent covariate of lower dimension $p$.
Furthermore, if $d\geq (p+1)\frac{(1-\rho)e}{\rho}$, we can bound $\Delta_{\rm miss}(d)$ from above and below,
\begin{align}
    \frac{\rho}{2e}(1-\rho)^{d-1}\Vert \beta^\star \Vert_2^2 & \leq \Delta_{\rm miss}(d)  \leq p\left(\frac{d\rho}{p(1-\rho)}\right)^p(1-\rho)^{d}\Vert \beta^\star \Vert_2^2,\label{eq:delta_miss_encadrement}
\end{align}
showing that $\Delta_{\rm miss}(d)$ decays exponentially fast with $d$ in the high-dimensional regime:  the impact of missing data on learning is therefore completely mitigated in high dimension.



\begin{theorem}\label{thm:imputationDD}
  Under \Cref{eq:modelRFgaussien,eq:modelYgaussien,ass:bernoulli_model}, the Bayes risk for predictors based on zero-imputed random features satisfies
  \begin{align}
  \label{eq:RimpUpperBound}
    & \esp \left[  R^\star_{\rm imp }(d) \right] - \sigma^2  \leq \left\{
        \begin{array}{ll}
         \inf\limits_{k\leq d} \left\{\frac{p-\rho k}{p} + \frac{ (1-\rho)\rho(k-1)}{p-\rho(k-1)-2}\frac{k}{p}\right\} \Vert \beta^\star \Vert_2^2  & \text{if } d< p,\\
         \frac{p}{\rho d+(1-\rho)p} \Vert \beta^\star \Vert_2^2 & \text{if } d \geq p.
        \end{array}
        \right. 
  \end{align}
Thus, when $d< p$, 
    \begin{align}
    \label{eq:DeltaimpmissLowDim}
                \Delta_{\rm imp/miss}(d)&\leq \frac{(1-\rho) \rho(d-1)}{p-\rho(d-1)-2}\frac{d}{p}\Vert \beta^\star \Vert_2^2\\
                &= \frac{\rho(d-1)}{p-\rho(d-1)-2}\Delta_{\rm miss}(d).\label{eq:DeltaimpmissLowDim3} 
    \end{align}
And, when $d\geq p$, 
\begin{align}\label{eq:DeltaimpmissHighDim1}
                \frac{(1-\rho)p}{\rho d+(1-\rho)p}\Vert \beta^\star \Vert_2^2 & \leq \Delta_{\rm imp/miss}(d)+\Delta_{\rm miss}(d)\\
                & \leq \frac{p}{\rho d+(1-\rho)p}\Vert \beta^\star \Vert_2^2.
                \label{eq:DeltaimpmissHighDim2}
\end{align}
\end{theorem}
Theorem \ref{thm:imputationDD} is the first result to provide a complete view of the impact of naive imputation on learning linear latent model.
In particular, it sheds light on the following low-dimensional behavior.
When $\rho d\ll p$, the error due to naive imputation appears to be negligible in comparison to the error $\Delta_{\rm miss}(d)$ due to missing data. Low-dimensional (missing) random features are unlikely to be strongly correlated, thus making imputation before training competitive (compared to the best predictor based on missing values). This is all the more true as the expected number of observed entries $\rho d$ is negligible compared to $p$.

In high dimensions where $d\gg p$, both errors $\Delta_{\rm imp/miss}(d)$ and $\Delta_{\rm miss}(d)$ vanish: neither the occurrence of missing data, nor their naive handling through imputation, hinder the learning task. 
This provides a refined and more rigorous analysis of this favorable behavior already identified in \citet{ayme2023naive}.
Remark that, contrary to $\Delta_{\rm miss}(d)$ (Equation~\eqref{eq:delta_miss_encadrement}), $\Delta_{\rm imp/miss}(d)$ does not seem to decrease exponentially as $d$ increases, but only as $1/d$. Not that, according to \eqref{eq:DeltaimpmissHighDim1}, this rate is optimal.



\paragraph{Illustration} We illustrate the bounds obtained in \Cref{prop:RiskcompletDD}, \Cref{prop:Riskmissing} and \Cref{thm:imputationDD} in \Cref{fig:dd}. 
In particular, we remark that the upper bounds \eqref{eq:RimpUpperBound} represented in (a) decrease with the number of features $d$ but is loose for $d$ close to and smaller than $p$. Indeed, for this regime, features are not enough isotropic to say that imputation by the mean (here by $0$) is relevant, and not enough correlated to exploit shared information between features. 
Figure (b) illustrates the shift point in $p$ for $R_{\rm mis}$ and $R_{\rm imp/mis}$ and the difficulties to learn with missing values (imputed or not) around $d=p$. 





\begin{figure*}[h]
    \centering
    \begin{tabular}{cc}   \includegraphics[width=0.49\textwidth]{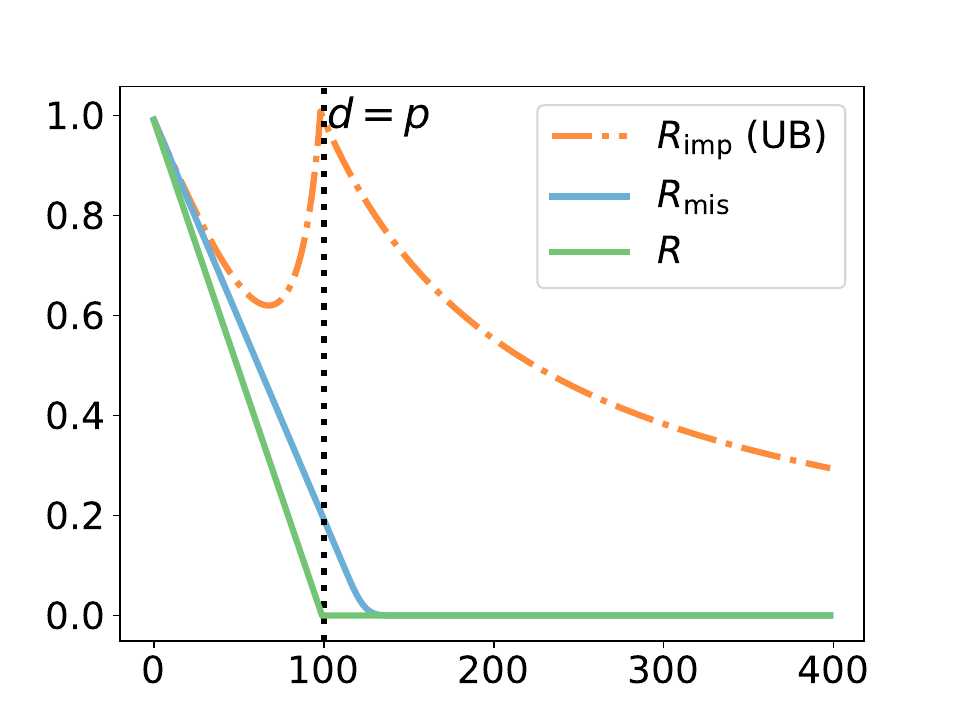} &     \includegraphics[width=0.49\textwidth]{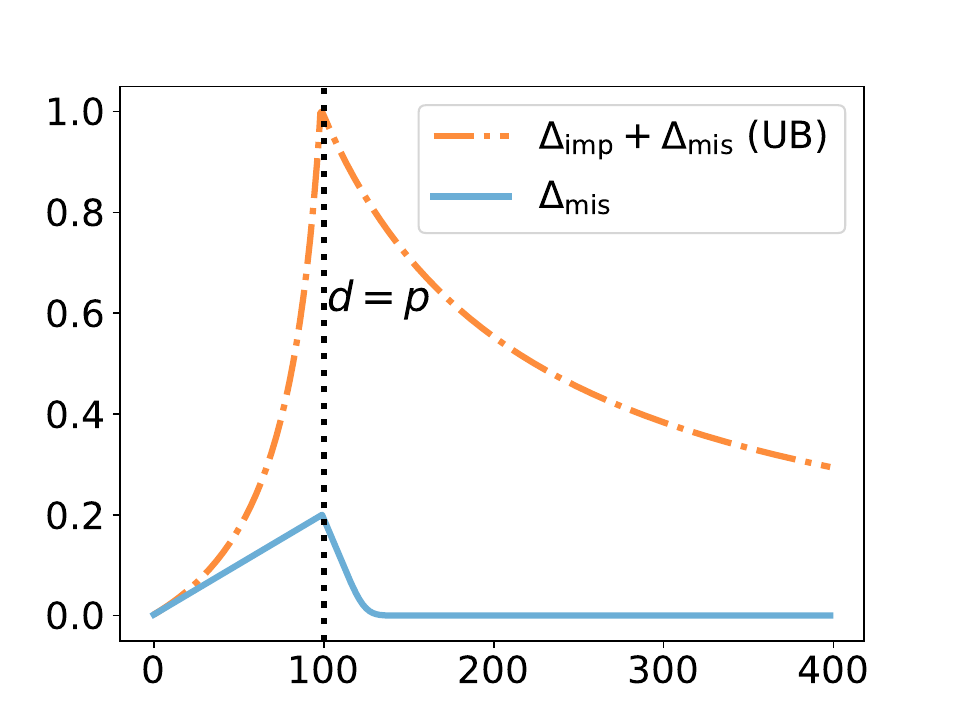} \\
    (a) & (b)
    \end{tabular}
    \caption{Evolution of different risks w.r.t.\ the number $d$ of random features, with $p=100$, $\Vert \beta^\star\Vert=1$, $\rho=0,8$ and $\sigma=0$. \label{fig:dd}}
\end{figure*}

\subsection{Learning from imputed data with SGD}\label{sec:SGDGaussian}
We leverage on our analysis of the Bayes risk of impute-then-regress strategy to propose a learning algorithm based on an SGD strategy. Our algorithm is computationally efficient, as it requires only one pass over the dataset, and shown to have theoretical guarantees.
Note that due to missing data, the model becomes miss-specified \citep[see][]{ayme2023naive}, a challenging study case that can still be handled by SGD procedures. 

\paragraph{SGD estimator.}
An averaged stochastic gradient descent (SGD) on the imputed dataset $(\tilde X_i)_{i\in [n]}$ is performed to directly minimize the theoretical risk $\theta\longmapsto R_{\rm imp}(\theta)$ over $\R^d$. The algorithm starts from  $\theta_0=0$ with a step-size $\gamma>0$, then follows the recursion 
\begin{equation}\label{eq:SGDiteration}
 \theta_{t}=\left[I-\gamma \tilde{X}_{t}\tilde{X}_{t}^{\top}\right]\theta_{t-1}+\gamma Y_{t}\tilde{X}_{t},   
\end{equation}
and outputs after $n$ iterations
the Polyak-Ruppert averaging $\bar\theta=\frac{1}{n+1}\sum_{t=1}^n \theta_{t}$, 
used to  estimate $\theta_{\imp}^\star$.
This algorithm (performing one pass over the training data) is optimal in terms of computational complexity. 
Note that the choice of the step size should depend intimately on the input dimension, with a slight variation, according to whether the setting is low or high-dimensional \citep[see,][for example]{dieuleveut2016nonparametric}. 

\begin{theorem}
\label{thm:sgd_rf_gaussian}
    Under \Cref{eq:modelRFgaussien,eq:modelYgaussien,ass:bernoulli_model}, for $d<p-1$, the SGD recursion with Polyak-Ruppert averaging and step size $\gamma=\frac{1}{d}$ satisfy
    \begin{align}
        &\esp [R_{\imp}(\bar\theta)-R_{\rm miss}^\star(d)]  
        \lesssim \frac{\rho(1-\rho)d(d-1)}{p(p-\rho (d-1)-2)}\Vert \beta^\star \Vert_2^2\notag\\
        &\quad \quad +\frac{d}{\rho n}\frac{d}{(p-\rho(d-1)-2)}\Vert \beta^\star \Vert_2^2
        +\rho\frac{d}{n}(\sigma^2+\Vert \beta^\star \Vert_2^2).\label{eq:GenGaussianLD}
    \end{align}
    For $d\geq p$, the choice $\gamma=\frac{1}{d\sqrt{n}}$ leads to
    \begin{align}\label{eq:GenGaussianHD}
        \esp[R_{\imp}(\bar\theta)-R^\star(d)] &\lesssim \frac{1}{\rho }\frac{p}{d} \left\Vert \beta^{\star}\right\Vert _{2}^{2} +\frac{p}{\rho^2(1-\rho)\sqrt{n}}\left\Vert \beta^{\star}\right\Vert _{2}^{2}+ \frac{\sigma^2+\Vert \beta^\star \Vert_2^2}{\sqrt{n}}.
    \end{align}
    
\end{theorem}
Regardless of the regime,  these generalization upper bounds are composed of three terms. The first one encapsulates the (deterministic) error due to the imputation of missing values. 
The last two are stochastic errors inherent to learning, corresponding respectively to the initial condition and the variance of the SGD recursion.

When $d<p-1$, the learning error 
\begin{align*}
\frac{d}{\rho n}\frac{d}{(p-\rho(d-1)-2)}\Vert \beta^\star \Vert_2^2
        +\rho\frac{d}{n}(\sigma^2+\Vert \beta^\star \Vert_2^2)
\end{align*}
decreases fast with the number $n$ of observation. However, the error due to the imputation of missing data  $\rho(1-\rho)\frac{d(d-1)}{p(p-(d-1)-2)}\Vert \beta^\star \Vert_2^2$ becomes negligible only for extremely low-dimensional regimes ($d\ll p$) and remains significant when  $d\lesssim p$. Therefore, for such a regime, even with a lot of observations, the imputation produces a large bias, and we recommend using other methods natively capable of handling missing values  \citep[e.g., dedicated tree-based methods, see][]{stekhoven2012missforest}.

In the regime $d\gg p$, the error $\frac{1}{\rho }\frac{p}{d} \left\Vert \beta^{\star}\right\Vert _{2}^{2}$ due to missing data and the zero-imputation procedure is low. Besides, the learning error $\frac{p}{\rho^2(1-\rho)\sqrt{n}}\left\Vert \beta^{\star}\right\Vert _{2}^{2}+ \frac{\sigma^2+\Vert \beta^\star \Vert_2^2}{\sqrt{n}}$ decreases at a (slow) rate $1/\sqrt{n}$. This slow rate of learning error is due to the fact that the covariance matrix of the imputed data $\Sigma_{\imp}=\esp[\tilde X\tilde X^\top]$ has a rank equal to $d$ and eigenvalues lower bounded by $\rho(1-\rho)$. Hence imputed data are not of low rank, even for $d\gg p$. However, the upper bound \eqref{eq:GenGaussianHD} becomes dimension-free for the regime $d>\rho(1-\rho)\sqrt{n}$. In this case, the bias due to missing data and zero imputation is negligible compared to the learning error. This gives a clear practical recommendation: if the observed rate is such that $\rho(1-\rho) < d/\sqrt{n} $, then zero-imputation does not deteriorate the learning procedure for $d$ large enough.

Overall, we have fully characterized the inherent error due to missing values in learning linear latent models, and proposed an efficient predictor based on SGD strategies. This section outlines that naive imputation thus remains a competitive and relevant technique not only for high-dimensional models but also for extremely low-dimensional ones, a favorable regime that was not identified so far in the literature. Note in passing that for the latter, the error bound \eqref{eq:GenGaussianLD} of the SGD predictor enjoys both a fast rate and a negligible approximation error, which can only be marginally improved.
However, the analysis conducted so far relies on strong assumptions (finite-dimensional latent model with Gaussian random features and uniform weights, linear model). 
In the next section, we propose an extension of our high-dimensional results to a more general framework.

\section{Extension to infinite-dimensional latent space}\label{sec:ExtensionRF}

In this section, we analyze the influence of missing data in learning, when the random feature model involves  an infinite-dimensional latent space.

\subsection{The extended random feature framework}


Consider a latent space $\mathcal{Z}$ (taking the place of $\R^p$), possibly of infinite dimension. We denote by $(Z_i)_{i\in[n]}$ i.i.d.\ latent variables, distributed as a generic random variable $Z\in\mathcal{Z}$.
We only observe the variables $(X_i)_{i\in[n]}$, i.i.d.\ copies of $X\in \mathcal{X} = \R^d$, resulting from the following transformation of the latent variables. 
\begin{assumption}[General random features]\label{ass:nongaussianRF}
The input variables are assumed to be given by
\begin{equation}\label{eq:RFmodel}
    X_{ij}=\psi (Z_i,W_j), \quad \text{for all } i\in[n] \text{ and } j\in [d],
\end{equation}
where the  weights $W_1, \hdots, W_d \in  \mathcal{W}$  are i.i.d.\ drawn according to a distribution $\nu$, and where $\psi : \mathcal{Z}\times \mathcal{W} \to \mathbb{R}$. Furthermore, we assume that $\psi(Z_i,.)\in L^2(\nu)$, and there exists $L>0$ such that 
$\esp[\psi(Z,W)^2|W]\leq L^2$ almost surely.
\end{assumption}

\Cref{ass:nongaussianRF} is an extension of \Cref{eq:modelRFgaussien} considering $\mathcal{Z}=\mathcal{W}= \R^p$, $\psi(Z_i,W_j)=Z_i^\top W_j$, $\nu$ the uniform distribution on $\mathbb{S}^{p-1}$ and $L^2=1$. But, such a setting of {general random features} encompasses many more scenarios, and has been extensively studied \citep{rahimi2007random}. 


In this framework, we aim to study the linear prediction of an output $Y$ given $X$, i.e., to build a prediction function of the form $g(X)= X^\top\theta$ with $\theta\in\mathbb{R}^d$. 
Note that this type of prediction can be also obtained as a function of 
 the (latent) variable $Z$, indeed,
\begin{equation}
    g(X) =  \sum_{j=1}^d \theta_j\psi(Z,W_j) =: f(Z).
\end{equation}
We can therefore define the corresponding class of functions with input space $\mathcal{Z}$ as
\begin{align*}\label{eq:defFnu}
\F^{(d)}_{\nu}:= \Big\lbrace f:\mathcal{Z} \to \R,  f(z) = \sum_{j=1}^d \theta_j\psi(z,W_j) , \theta\in\R^d \Big\rbrace
\end{align*}
Note that the class $\F^{(d)}_{\nu}$ is random because the weights $(W_j)_{j\in [m]}$ themselves are random. 
When the number $d$ of random features tends to infinity, we can define the set $\F^{(\infty)}_\nu$ of functions which take the form:
\begin{equation}
    f(Z) = \int \alpha_f(w)\psi(Z,w)d\nu(w), 
\end{equation}
for any $\alpha_f \in L^2(\nu)$. The associated norm is given by  
\begin{align}\label{def:nu-norm}
    \left\Vert f\right\Vert _{\nu}^{2}:=\inf_{\alpha\in\mathbb{L}_{2}(\nu)} \int\left|\alpha(w)\right|^{2}d\nu(w)\quad \ensuremath{\quad}{\rm s.t.}\quad\forall z\in\mathcal{Z},f(z)=\int\alpha(w)\psi(z,w)d\nu(w).
\end{align}
This norm corresponds to an RKHS norm, we refer the interested reader to \citet{bach2017breaking} for further details. 
We denote by $R^\star(\infty)$, the risk of the best predictor belonging to the class $\F^{(\infty)}_\nu$, i.e., $R^\star(\infty)= \inf_{f\in\F^{(\infty)}_\nu} \esp \left[(Y-f(Z))^2\right]$.

The setting considered in this section includes, for instance, Fourier random features \citep{rahimi2007random,rudi2017generalization}.
\begin{example}[Fourier random
features]\label{ex:Fourierfeatures} Consider $Z\in\mathcal{Z}$ where $\mathcal{Z}$ is a compact subset of $\R^p$, $W=(A,B,C)\in\mathcal{W}= \R^p\times \R\times \{-1,1\}$ and fix
\begin{equation*}
    \psi(Z,W)=\cos(A^\top Z+B)+2C,
\end{equation*}
with $A\sim\mathcal{N}(0,I)$, $B\sim\mathcal{U}([0,2\pi])$ and $C\sim\mathcal{U}(\{-1,1\})$. 
Note that \Cref{ass:nongaussianRF} holds here with $L^2=3$. 
The resulting function class $\F_{\nu}^{(\infty)}$ described by these random features is dense for $\Vert\cdot \Vert_{\infty}$ in the space of continuous functions.
\end{example}
 As shown in this example, with a proper choice of $\nu$ and $\psi$, the class $\F_{\nu}^{(\infty)}$ can approximate any function that makes the following assumption feasible. 
\begin{assumption}\label{ass:bayes_predictor}
    The Bayes predictor $f^\star (z) = \esp{[Y|Z=z]}$ belongs to $\mathcal{F}_{\nu}^{(\infty)}$. 
\end{assumption}
Under Assumption \ref{ass:bayes_predictor},
the model $Y=f^\star(Z)+\epsilon$, is well defined, i.e., $\esp[\epsilon|Z]=0$. Remark that \Cref{ass:bayes_predictor} can be seen as a natural extension of linear model \Cref{eq:modelYgaussien}.  

\subsection{Impact of missing data and imputation in the RF framework}
The general random features $(X_i)_{ i \in [n]}$ are assumed to be corrupted by MCAR entries, whether during training and test phases. 
Our goal is to study the quantity $R^\star_{\rm imp}(d)$. Note that \eqref{eq:riskordre} can be rewritten here as 
\begin{equation*}
    R^\star(\infty)\leq R^\star(d)\leq R_{\rm miss}^\star(d)\leq R_{\rm imp}^\star(d).
\end{equation*}
Thus, introducing the quantity
\begin{align}
    &\Delta_{\rm imp}^{(\infty)}(d):= \esp R_{\rm imp}^\star(d)-R^\star(\infty)\notag\\
    &= \Delta_{\rm miss}(d)+\Delta_{\rm imp/miss}(d)+\esp R^\star(d)-R^\star(\infty),\label{eq:delta_imp_general_RF}
\end{align}
we encapsulate (i) the error $\esp R^\star(d)-R^\star(\infty)$  due to learning from a finite number of random features, (ii)  the error $\Delta_{\rm miss}(d)$ due to learning with missing inputs and, (iii)   the approximation error $\Delta_{\rm imp/miss}(d)$ due to the imputation by zero. 
\begin{theorem} \label{thm:DeltaExtension}
Under \Cref{ass:bernoulli_model,ass:nongaussianRF},
\[
\Delta_{\rm imp}^{(\infty)}(d)\leq\inf_{f\in\F_\nu^{(\infty)}}\left\{ R(f)-R^\star(\infty)+\frac{\lambda_{\rm imp}}{d}\left\Vert f\right\Vert _{\nu}^{2} \right\},
\]
with $\lambda_{\rm imp} = \frac{L^2}{\rho}$. In particular, under \Cref{ass:bayes_predictor}, 
\[
\Delta_{\rm imp}^{(\infty)}(d)\leq\frac{\lambda_{\rm imp}}{d}\left\Vert f^\star\right\Vert _{\nu}^{2} .
\]
\end{theorem}
\Cref{thm:DeltaExtension} provides an upper bound on $\Delta_{\rm imp}^{(\infty)}(d)$ for general random feature models.  
In particular, the latter can be compared to a ridge bias when performing a kernel ridge regression in $\F^{(\infty)}_\nu$, and choosing the penalization strength of the order of $\lambda_{\rm imp}/d$. 
Furthermore, under \Cref{ass:bayes_predictor} (well-specified model), this bias converges to zero with a rate of $\left\Vert f^\star\right\Vert _{\nu}^{2}/(\rho d)$. By applying this result to a finite-dimensional latent model under \Cref{eq:modelRFgaussien,eq:modelYgaussien}, and remarking that $\left\Vert f^\star\right\Vert _{\nu}^{2}=p\left\Vert \beta^\star\right\Vert _{2}^{2}$, we recover the same rate $p/(\rho d)\left\Vert \beta^\star\right\Vert _{2}^{2}$ exhibited in \Cref{thm:imputationDD}. According to \Cref{thm:imputationDD}, this rate cannot be improved in general. 
More globally, missing data in RF models become harmless when learning with a large number of random features. 
It should be noted that when \Cref{ass:bayes_predictor} does not hold anymore, one can still conclude that the  bias $\Delta_{\rm imp}^{(\infty)}$
tends to zero but at an arbitrarily slow rate. 
Regarding \Cref{ass:nongaussianRF}, it remains a mild requirement; in particular, it does not require centered inputs. This underlines that there is no need to impute by the mean to obtain  $\Delta_{\rm imp}^{(\infty)}$ converging to $0$ in high-dimensional regimes. 

\subsection{SGD generalization upper bound}

In this section, we assess the generalization performance of the SGD iterates when working with an underlying general random feature model. 

\begin{assumption}\label{ass:momentinf}
    There exists $\ell>0$ such that, almost surely
    \begin{equation*}
        \esp[\psi(Z,W)^2|W]\geq \ell^2.
    \end{equation*}
\end{assumption}
This assumption holds when features are renormalized (i.e., when $\ell=L=1$) or in the case of random Fourier features (see \Cref{ex:Fourierfeatures}) with $\ell=1$.

\begin{assumption}\label{ass:kurtosis}
Assume that almost surely,
    \begin{equation*}
        \vert\psi(Z,W)\vert^2\leq \kappa L^2.
    \end{equation*}
    
\end{assumption}
This assumption is satisfied in \Cref{ex:Fourierfeatures} with $\kappa=2$ and $L^2=3$.

\begin{theorem}
\label{thm:genExtension}
Under the general framework covered by \Cref{ass:nongaussianRF,ass:bayes_predictor,ass:momentinf,ass:kurtosis} with MCAR data (Assumption \ref{ass:bernoulli_model}), the SGD recursion with Polyak-Ruppert averaging and step size $\gamma=1/(\kappa d \sqrt{n})$  satisfy
\begin{align}
    \label{eq:GenNonGaussianHD}
        \esp[R_{\imp}(\bar\theta)-R^\star(\infty)] \lesssim \frac{L^2}{\rho d} \left\Vert f^{\star}\right\Vert _{\nu}^{2} +\frac{L^2}{\ell^2}\frac{L^2}{\rho^2(1-\rho)\sqrt{n}}\left\Vert f^{\star}\right\Vert _{\nu}^{2}+ \frac{\kappa L^2\esp Y^2}{\sqrt{n}}.
\end{align}
\end{theorem}
Theorem~\ref{thm:genExtension}  outlines that, even for very general random features model (with possibly a latent space of infinite dimension),
the impact of (i) the finite number of features, (ii) the missing data, and  (iii) the imputation by $0$, represented by the quantity $\frac{L^2}{\rho d} \left\Vert f^{\star}\right\Vert _{\nu}^{2}$, remains negligible in high dimension. Similarly to \eqref{eq:GenGaussianHD}, the learning error $\frac{L^2}{\ell^2}\frac{L^2}{\rho^2(1-\rho)\sqrt{n}}\left\Vert f^{\star}\right\Vert _{\nu}^{2}+ \frac{\kappa L^2 \esp Y^2}{\sqrt{n}}$ decreases with a slow rate. More precisely, when $d\gg \frac{L^2}{\ell^2}\rho(1-\rho)\sqrt{n}$, the upper bound is dimension free and the bias $\Delta_{\rm imp}^{(\infty)}$ due to imputation becomes completely negligible. Note that, for renormalized features ($L^2=l^2=1$), the transition from a low-dimensional regime to a high-dimensional one is given by $d= \rho(1-\rho)\sqrt{n}$ (as for \Cref{thm:sgd_rf_gaussian}), which is very easy to evaluate. 

\section{Beyond the MCAR assumption}\label{sec:MNAR}
To go beyond the MCAR missing data framework used in the previous section, we now consider missing not at random (MNAR) data, in which the missingness indicator of any variable can depend on the (possibly missing) value of the variable. 
In particular, we assume that the missing patterns $(P_i)_i$ depend on the latent features $(Z_i)_i$, which results in a MNAR scenario. 

\begin{assumption}\label{ass:RFmodelNAR}
    Suppose that $P$ and $Y$ are independent. Furthermore, consider that there exists  an i.i.d. sequence $(\uline{W}_j')_{j\in[d]}$ i.i.d.\ drawn according to some distribution $\mu$ supported on a set $\mathcal{W'}$ and assume that $P_1,\dots,P_d|Z,W_1',\dots,W_d'$ are independent. We assume that the sequences $(\uline{W}_j')_{j\in[d]}$ and $(W_j)_{j\in[d]}$ are independent and that
\begin{equation*}
    \mathbb{P}\left(P_j|Z,\uline{W}_j'\right)=\phi (Z,\uline{W}_j'), \quad \text{for all } j\in [d],
\end{equation*}
 where $\phi : \mathcal{Z}\times \mathcal{W'} \to (0,1]$ is a continuous function. 
\end{assumption}

The following result shows that the asymptotic property of   $R^\star_{\rm imp}$ in a MCAR setting (\Cref{thm:imputationDD}) remains valid in the MNAR setting of \Cref{ass:RFmodelNAR}.
\begin{theorem}\label{thm:convergenceMNAR}
Under \Cref{ass:RFmodelNAR}, consider one of the following settings:
    \begin{enumerate}
        \item[$(i)$] (finite-dimensional latent space) Under \Cref{eq:modelRFgaussien,eq:modelYgaussien}, assume the distribution of the missing mechanism to be given for $\uline{W}'=(W_0',W')\in\R\times\R^d$, by $\phi (Z,\uline{W}')=\Phi(Z^\top W'+W_0')$ with $\Phi$ a Lipschitz function. Additionally, $0$ is required to belong to the support of $W'$.
        \item[$(ii)$] (general latent space) Under \Cref{ass:nongaussianRF,ass:bayes_predictor}, assume in addition that $\mathcal{Z}$ is compact, $f^\star$  continuous and $\mathcal{F}^{(\nu)}$ dense in the space of continuous functions equipped with the norm $\Vert \cdot \Vert_{\infty}$. 
    \end{enumerate}
Then, almost surely,
\begin{equation*}
    \lim_{d\to +\infty} R^\star_{\rm imp}(d)= R^\star(\infty).
\end{equation*}  
As a consequence, $$\lim_{d}\Delta_{\rm imp}^{(\infty)}(d)=\lim_{d}\Delta_{\rm miss}(d)=\lim_{d}\Delta_{\rm miss/imp}(d)=0.$$
\end{theorem}
This result shows that the benign impact of missing data and  imputation on predictive performances in high dimension holds true outside the MCAR assumption, even for missing scenarios (MNAR) often considered as more challenging. Let us consider two non-trivial examples.
\begin{example}[Gaussian random features with logistic model]\label{ex:logisticNAR}
Consider the finite-dimensional latent model of \Cref{eq:modelRFgaussien,eq:modelYgaussien}, where $\mathcal{Z}= \R^p$ and $\uline{W}_j'=(W_{0j}',W_j')\in\R\times\R^d$, and assume that the conditional distribution of the missing patterns $P_j$ 
is given by
\begin{equation*}
    \mathbb{P}\left(P_j|Z,\uline{W}_j'\right)=\Phi(W_{0j}'+W_j'^\top Z)= \frac{1}{1+e^{W_{0j}'+W_j'^\top Z}}.
\end{equation*}
In this example, the features $X_j$ are assumed to be missing according to a logistic model on the latent variables $Z$. In this setting, we can show that \Cref{thm:convergenceMNAR} $(i)$ applies, since in particular, $0$ belongs the support of $W'_j$. Note that, if $W_j'=0$ almost surely then this model corresponds to a MCAR scenarios but with different proportion of missing values for each feature. The model is no longer MCAR as soon as random variable $W_j'$ is not exactly equal to $0$. 
\end{example}

\begin{example}[Fourier random features for any function $\phi$]
Let us consider the framework of \Cref{ex:Fourierfeatures} with a continuous function $f^\star$. Then \Cref{thm:convergenceMNAR} $(ii)$ applies for any continuous function $\phi$ (in particular, we can consider the logistic model of \Cref{ex:logisticNAR} without any condition on $W'$).
\end{example}
For these two MNAR examples, $$\lim_{d}\Delta_{\rm imp}^{(\infty)}(d)=\lim_{d}\Delta_{\rm miss}(d)=\lim_{d}\Delta_{\rm miss/imp}(d)=0,$$
which means that missing values and imputations vanishes with the dimension.

\section{Conclusion}
Thanks to the rigorous framework of random features models, we prove that naive imputation is relevant both in high- and low-dimensional regimes. 
In particular, the bias induced by imputation is negligible compared to the one induced by missing data, therefore showing that zero-imputation strategies may lead to near-optimal predictors. 
Following this analysis, we prove that an SGD procedure trained on zero-imputed data reaches near-optimal rate of consistency in low-dimensional regimes, but still suffer from slow rates in high-dimensional ones. Obtaining fast rates for the latter setting is still an open and interesting question.
Whilst our analysis extends beyond the MCAR scenario, rates of consistency for SGD procedures remain to be derived for such settings.

\bibliography{references}
\bibliographystyle{plainnat}

\newpage
\appendix
\onecolumn

\section{Notations}\label{app:notations}

For two vectors (or matrices) $a,b$, we denote by $a\odot b$ the Hadamard product (or component-wise product). $[n]=\left\{1,2,...,n\right\}$. For two symmetric matrices $A$ and $B$, $A\preceq B$ means that $B-A$ is positive semi-definite. The symbol $\lesssim$ denotes the inequality up to a universal constant.  Table~\ref{tab:notations} summarizes the  notations used throughout the paper and appendix.

\begin{table}[h!]
    \centering
        \caption{Notations}
            \label{tab:notations}

    \begin{tabular}{l l} \toprule
         $\Vert u\Vert_M^2$ & $u^\top M u$ the semi-norm induced by a positive matrix $M$ \\
         $\Vert M\Vert_{\mathrm{Fr}}^2$ & The Frobenius norm of $M$ \\
         $\mathrm{Tr}(M)$ & The sum of diagonal elements of $M$ \\
         $M+\lambda$ & Abuse of notation for $M+\lambda I_p$ \\
         $M^\dagger$ & The Moore-Penrose pseudoinverse of $M$ \\
         $\mathbb{S}_p$ & The unit sphere of $\R^p$\\
         $\mathrm{Span}(u_j,j\in[k])$& The linear span induced by $(u_j)_{j\in[k]}$ \\
         $P$ & The mask\\ 
         $\mathbf{W}$ & $(W_1,\dots,W_d)^\top$ the matrix of weights\\
         $R(\theta)$ & the risk of linear predictor $\theta$ on complete data\\
         $R_{\rm imp}(\theta)$ & the risk of linear predictor $\theta$ on imputed data\\
         $\theta^\star$ &Best linear predictor on complete data \\
         $\theta_{\imp}^\star$&Best linear predictor on imputed data\\

          $\Sigma$ & $\esp [XX^\top|\mathbf{W}]$  \\
          $\lambda_j$ & eigenvalues of $\Sigma$  \\
          $u_j$ & eigendirections of $\Sigma$  \\

         $\rho$ & Theoretical proportion of observed entries \\
         $L^2(\nu)$ & The set of two square $\nu$ integrable functions  \\
         \bottomrule
    \end{tabular}
    \end{table}

\section{Preliminary results - random matrices}
\label{app:RandomMatrix}
We provide here a reminder on singular values decomposition and Moore-Penrose pseudoinverse. We can found these results and more on linear algebra in \citet[][appendix]{giraud2021introduction}. 
\begin{theorem}\label{thm:svd}
   Any $n\times p$ real-valued matrix of rank $r$ can be decomposed as 
   \begin{equation*}
       A=\sum_{j=1}^r \sigma_j u_j v_j^\top,
   \end{equation*}
   where 
   \begin{itemize}
       \item $\sigma_1\geq\dots\geq\sigma_r>0$,
       \item $(\sigma_1,\dots,\sigma_r)$ are the nonzero eigenvalues of $A^\top A$ and $AA^\top$, and 
       \item $(u_1,\dots,u_r)$ and $(v_1,\dots,v_r)$ are two orthonormal families of $\R^n$ and $\R^p$, such that $AA^\top u_j=\sigma_j^2u_j$ and $A^\top Av_j=\sigma_j^2v_j$. 
   \end{itemize}
   Furthermore, the Moore-Penrose pseudo inverse defined as 
   \begin{equation*}
       A^\dagger=\sum_{j=1}^r \sigma_j^{-1} v_j u_j^\top,
   \end{equation*}
   satisfied
   \begin{enumerate}
       \item $A^\dagger A$ is the orthogonal projector on lines of $A$,
       \item $A A^\dagger $ is the orthogonal projector on columns of $A$,
       \item $(AO)^\dagger=O^\top A^\dagger$ for any orthogonal matrix $O$.
   \end{enumerate}
\end{theorem}

For any function $f:\R\longmapsto\R$, and any positive matrix $A\in\R^{d\times d}$ with the following spectral decomposition $A=\sum_{j=1}^d \lambda_j v_jv_j^\top$, we denote by $f(A)$ the matrix corresponding to the spectral decomposition 
\begin{equation*}
    f(A):=\sum_{j=1}^d f(\lambda_j) v_jv_j^\top.
\end{equation*}
\begin{theorem}[Jensen inequality for random positive matrix, see Theorem 2.10 in \cite{carlen2010trace}]\label{thm:JensenTrace}
    Let $A$ be a random positive matrix. For all convex functions $f$, we have  
    \begin{equation*}
        \mathrm{Tr}(f(\esp A))\leq\esp\mathrm{Tr}(f(A)).
    \end{equation*}
\end{theorem}
\begin{propo}\label{prop:exchangeabilityResults}
    Let $A= \sum_{j=1}^d Z_jZ_j^\top$ with $Z_1,\dots,Z_d$ i.i.d. random vector of $\R^p$ with $\Vert Z_1\Vert_2^2\leq 1$ almost surely and $\esp ZZ^\top=\alpha I_p$, then, for all $\lambda >0$
    \begin{equation}
        \frac{p}{d\alpha+\lambda}\leq \esp\mathrm{Tr}\left(( A +\lambda I_p)^{-1}\right)\leq (1+1/\lambda)\frac{p}{d\alpha+\lambda}.
    \end{equation}
\end{propo}

\begin{proof}[Proof of \Cref{prop:exchangeabilityResults}]
This result is a direct application of \citet[][Lemma 2]{mourtada2022elementary} considering $\hat\Sigma=\frac{1}{d}A$
\end{proof}

\begin{lemma}\label{lem:vector_random_matrice}
Let $M\in\R^{p\times p}$ be a random symmetric matrix, such that for all vectors  $u,v\in \mathbb{S}^{p-1}$, 
$\mathrm{Law}(u^\top M u)=\mathrm{Law}(v^\top M v)$. Then, for all $\beta\in\R^p$,
\begin{equation*}
    \esp \left[ \beta^\top M\beta \right] = \Vert\beta\Vert_2^2\frac{\esp \mathrm{Tr}(M)}{p}.
\end{equation*}
This is in particular satisfied if, for any orthogonal matrix $O$, $OMO^\top$ has the same law as $M$. 
\end{lemma}
\begin{proof}
By assumption, for all $u,v \in \mathbb{S}^{d-1}$, $\esp u^\top M u = \esp v^\top M v$. Thus, there exists $\alpha$ such that, for all $v\in \mathbb{S}^d$, $v^\top \esp M v=\esp v^\top M v = \alpha$, which entails that $\esp M=\alpha I$ by characterization of symmetric matrices. Therefore, $\esp \mathrm{Tr}(M)= \mathrm{Tr}(\esp M)=p\alpha$, and $\esp M=\frac{\esp \mathrm{Tr}(M)}{p}I$. Hence, for all $\beta\in\mathbb{R}^p$
\begin{equation*}
    \esp \left[ \beta^\top M\beta \right] =   \beta^\top \esp M\beta =  \Vert\beta\Vert_2^2\frac{\esp \mathrm{Tr}(M)}{p}.
\end{equation*}
The last point easily follows, see for example \citet[][Proposition 2.14]{page1984multivariate}  for the case of invariant distributions by orthogonal transforms. 

\end{proof}

The following result is inspired by the result of \citet{cook2011mean}, that is an adaptation of that of \citet{von1988moments}. 
\begin{lemma}\label{lem:RMfrobenius}
For all $0<d<p-1$, let $\mathbf{W}\in\R^{p\times d}$ such that columns of $\mathbf{W}$ are i.i.d. and uniform over $\mathbb{S}^{p-1}$, 
then 
\begin{equation*}
    \esp \Vert\mathbf{W}^\dag\Vert_{\rm Fr}^2=d\left(1+\frac{d-1}{p-d-1}\right). 
\end{equation*}
\end{lemma}
\begin{proof}

Up to a polar coordinate change of variable, one can show that the distribution of the columns of $\mathbf{W}$ corresponds to that of normalized Gaussian vectors, i.e., for all $j \in [d]$, 
    \begin{equation*}
        W_j= \frac{G_j}{\Vert G_j\Vert_2},
    \end{equation*}
where $(G_j)$ are i.i.d. of law $\mathcal{N}(0,I_p)$. 
Note that the columns of $ \mathbf{W}$ are the rows of $\mathbf{M} = \mathbf{W}^\top$. As $d<p$, 
\begin{equation*}
    \mathbf{M}\mathbf{M}^\dag = I_d. 
\end{equation*}
because the rows of $\mathbf{M}$ are almost surely linearly independent. For all $j \in [d]$, we let $l_j$ be the $j$-th row
of $\mathbf{M}$, and $c_j$ the $j$-th column of $\mathbf{M}^\dag$. Therefore, for all $k\neq j$, $l_k^\top c_j=0$, then $c_j\in \mathrm{Span}(l_k, k\neq j)^\perp$. 
Note, that $\mathrm{Span}(l_j,j\in[d])= \mathrm{Span}(c_j,j\in[d])$ by property of Moore-Penrose pseudoinverse (\Cref{thm:svd}). Thus, $c_j$ as the form, 
\begin{equation*}
    c_j= \theta_j P_j l_j, 
\end{equation*}
where $P_j$ is the orthogonal projection on $\mathrm{Span}(l_k, k\neq j)^\perp$. Besides, $l_j^\top c_j=1$ gives us that 
\begin{equation*}
    \theta_j = \frac{1}{\Vert P_j l_j\Vert_2^2}. 
\end{equation*}
Thus, 
\begin{equation}\label{eq:Mfrob}
    \Vert M^\dag\Vert_{\rm Fr}^2= \sum_j \Vert c_j\Vert_2^2= \sum_j \frac{1}{\Vert P_j l_j\Vert_2^2}
\end{equation}
As $l_j= W_j= \frac{G_j}{\Vert G_j\Vert_2}$, we can write 
\begin{equation*}
    \frac{1}{\Vert P_j l_j\Vert_2^2}=\frac{\Vert G_j\Vert_2^2}{\Vert P_j G_j\Vert_2^2}
\end{equation*}
Using that $\Vert G_j\Vert_2^2= \Vert P_jG_j\Vert_2^2+\Vert (I_p-P_j)G_j\Vert_2^2$, we have 
\begin{equation*}
    \frac{1}{\Vert P_j l_j\Vert_2^2}=1+\frac{\Vert (I_p-P_j)G_j\Vert_2^2}{\Vert P_j G_j\Vert_2^2}.
\end{equation*}   
Conditioning by $(G_k)$ with $k\neq j$, and using Cochran theorem $(I_p-P_j)G_j|G_k,k\neq j$ and $P_j G_j|G_k,k\neq j$ are two independent standard normal vector of respective dimensions $p- (p-d+1)=d-1$ and $p-d+1$. Thus, 
\begin{align}
    \esp\left[ \frac{1}{\Vert P_j l_j\Vert_2^2}|G_k,k\neq j \right] &= 1+ \esp\left[ \Vert (I_p-P_j)G_j\Vert_2^2|G_k,k\neq j \right]\esp\left[ \frac{1}{\Vert P_jG_j\Vert_2^2}|G_k,k\neq j \right]\\
    &= 1+ \frac{d-1}{p-d-1},
\end{align}
because $\esp\left[ \Vert (I_p-P_j)G_j\Vert_2^2|G_k,k\neq j \right]= d-1$ and $\esp\left[ \frac{1}{\Vert P_jG_j\Vert_2^2}|G_k,k\neq j \right]= \frac{1}{p-d-1}$ as the expectation of an inverse-chi-squared of parameter $p-d+1 > 2$  (with $d<p-1$).
Then, taking the expectation of \eqref{eq:Mfrob} leads to the result,
\begin{equation*}
    \esp \Vert M^\dag\Vert_{\rm Fr}^2=d+\frac{d(d-1)}{p-d-1}.
\end{equation*}
\end{proof}

\begin{lemma}\label{lem:hadamard_monotonicity}
   Let $A,B,V$ three symetrics non-negative matrix, if $A\preceq B$
then $A\odot V\preceq B\odot V$. 
\end{lemma}
 \begin{proof}
     Let $X\sim\mathcal{N}(0,V)$ and $\theta\in\R^d$,

\begin{align*}
\left\Vert \theta\right\Vert _{A\odot V}^{2} & =\theta^{\top}A\odot V\theta\\
 & =\theta^{\top}\left(\left(\esp XX^{\top}\right)\odot A\right)\theta\\
 & =\esp\left[\theta^{\top}\left(\left(XX^{\top}\right)\odot A\right)\theta\right]\\
 & =\esp\left[\sum_{i,j}\theta_{i}\left(\left(XX^{\top}\right)\odot A\right)_{i,j}\theta_{j}\right]\\
 & =\esp\left[\sum_{i,j}\theta_{i}X_{i}X_{j}A_{i,j}\theta_{j}\right]\\
 & =\esp\left[\sum_{i,j}\left(\theta_{i}X_{i}\right)\left(\theta_{j}X_{j}\right)A_{i,j}\right]\\
 & =\esp\left[\left\Vert X\odot\theta\right\Vert _{A}^{2}\right]\\
 & \leq\esp\left[\left\Vert X\odot\theta\right\Vert _{B}^{2}\right]\\
 & =\left\Vert \theta\right\Vert _{B\odot V}^{2}
\end{align*}
 \end{proof}

\section{Proof of \Cref{sec:DD}}\label{app:RFfinite}

The following result, established by \citet{ayme2023naive}, is used to derive an expression of  $\Delta_{\rm missing}+\Delta_{\rm imp/miss}$.
\begin{lemma}[Proposition 3.1 of \citep{ayme2023naive}]
\label{lem_previous_work1}
    For all $\theta\in\R^d$, 
    \begin{equation}\label{eq:BimpDecompoBV}
    R_{\rm imp}(\theta)= R(\rho\theta)+\rho(1-\rho)\Vert\theta\Vert_{{\rm diag}(\Sigma)}^2.    \end{equation}
\end{lemma}
Recalling that 
\begin{align}
 \Delta_{\rm miss}+\Delta_{\rm imp/miss} =    \esp \left[ R_{\rm imp}^\star(d)-R^\star(d)\right], 
\end{align}
we deduce from \Cref{lem_previous_work1} that 
\begin{equation}\label{eq:BiasRidgeForm}
    \Delta_{\rm miss}+\Delta_{\rm imp/miss}=\esp \inf_{\theta\in\R^d}\left\{R(\theta)-R^\star(d)+\frac{1-\rho}{\rho}\Vert\theta\Vert_{{\rm diag}(\Sigma)}^2\right\}.
\end{equation}
Additionally, when ${\rm diag}(\Sigma)=I_p$ (in particular for model \eqref{eq:modelRFgaussien}), by optimization, we obtain, 
\begin{equation}\label{eq:BiasRidgeFormClose}
    \Delta_{\rm miss}+\Delta_{\rm imp/miss}=\lambda\esp \Vert\theta^\star\Vert_{\Sigma(\lambda I+\Sigma)^{-1}}^2,
\end{equation}
with $\lambda=\frac{1-\rho}{\rho}$.

\begin{lemma}\label{lem:decompo_pat_esperance}
Under \Cref{ass:bernoulli_model},
\begin{equation*}
    \esp R_{\rm miss}^\star(d)=\sum_{k=0}^d \P(B=k)\esp R^\star(k),
\end{equation*}
where $B$ is a binomial random variable of parameters $d$ and $\rho$. 
\end{lemma}
\begin{proof}
    Using the decomposition of the Bayes predictor  from \citet{le2020linear}, we have
    \begin{equation}
        R_{\rm miss}^\star(d)= \sum_{m\in\{0,1\}^d} \P(P=m)R_{m}^\star,
    \end{equation}
where 
\begin{equation*}
    R_{m}^\star=  \inf_{f} \esp \left[(Y-f(X_{{\rm obs}(m)}))^2|P=m,W_1,\dots,W_d\right],
\end{equation*}
is the local Bayes risk given $(P=m)$. Using MCAR assumption (\Cref{ass:bernoulli_model}), and Gaussian assumption, according to \citet{le2020linear}, each local Bayes predictor are linear, thus 
\begin{equation*}
    R_{m}^\star=  \inf_{\theta } \esp \left[(Y-\theta^\top X_{{\rm obs}(m)}))^2|W_j, j\in {\rm obs}(m)\right]. 
\end{equation*}
As $(W_j)$ are i.i.d. (and independent of $Y$), $R_{m}^\star$  has the same law as $R^\star(|m|)$ 
where $|m|$ is the number of observed components of $m$. Thus, 
\begin{equation}
        \esp R_{\rm miss}^\star(d)= \sum_{m\in\{0,1\}^d} \P(P=m)\esp R^\star(|m|)). 
\end{equation}
Grouping the missing patterns of the same size, we conclude that,
\begin{equation*}
    \esp R_{\rm miss}^\star(d)=\sum_{k=0}^d \P(B=k)\esp R^\star(k),
\end{equation*}
where $B$ is a binomial law of parameters $d$ and $\rho$.

\end{proof}

\subsection{Proof of Proposition \ref{prop:RiskcompletDD}}

By definition, 
\begin{align*}
R^\star(d)&= \esp\left[(X^\top\theta^\star -Y)^2 | W_1, \hdots, W_d \right]\\ 
&=\esp\left[(X^\top\theta^\star -Z^\top \beta^\star-\epsilon)^2 | W_1, \hdots, W_d \right]& \tag{using \eqref{eq:modelYgaussien}}\\
&=\sigma^2+\esp\left[(X^\top\theta^\star -Z^\top \beta^\star)^2 | W_1, \hdots, W_d \right],  
\end{align*}
using that $\epsilon$ is an independent noise of variance $\sigma^2$. 
We have $X^\top\theta^\star= Z^\top\sum_j \theta_j^\star W_j$. Then, 
\begin{align*}
R^\star(d)&=\sigma^2+\esp\left[\left(\left(\sum_{j=1}^d\theta_j^\star W_j-\beta^\star)\right)^\top Z\right)^2 \Bigg| W_1, \hdots, W_d  \right]\\
&=\sigma^2+\left\Vert\beta^\star -\sum_{j=1}^d \theta_j^\star W_j\right\Vert_2^2, 
\end{align*}
by isotropy of $Z$ ($Z\sim \mathcal{N}(0,I_p)$ and thus $\esp ZZ^\top=I$).
Using that $\sum_{j=1}^d \theta_j^\star W_j$ belongs to $ \mathrm{Span}(W_1,\dots,W_d)$, we get that  $\sum_{j=1}^d \theta_j^\star W_j=P_d\beta^\star$ where $P_d$ is the orthogonal projection on 
$ \mathrm{Span}(W_1,\dots,W_d)$. Then, 
\begin{align*}
R^\star(d)= \sigma^2+\Vert (I-P_d)\beta^\star\Vert_2^2=\sigma^2+(\beta^\star)^\top(I-P_d)\beta^\star . 
\end{align*}
Remark that $P_d$ is a random matrix (since $W_1,\dots,W_d$ are random). Denoting by $\mathbf{W}$ the matrix admitting $W_1,\dots,W_d$ as rows, the projection matrix can be rewritten as $P_d= \mathbf{W}^\dagger \mathbf{W}$. Thus, for all orthogonal matrix $O$, $OP_dO^\top=(\mathbf{W}O^\top)^\dagger \mathbf{W}O^\top$. The matrix $\mathbf{W}O^\top$ has rows $O^\top W_1,\dots,O^\top W_d$, which is an i.i.d.\ sequence of random vectors on the unit sphere (since $O^\top$ is an orthogonal). Indeed, $\mathbf{W}O^\top$ and $\mathbf{W}$ have the same distribution, in consequence $P$ and $OP_dO$ have the same distribution too. Thus, by \Cref{lem:vector_random_matrice}, if $d<p$, 
\begin{align*}
\esp R^\star(d)&= \sigma^2+\frac{1}{p}\Vert\beta^\star\Vert_2^2\esp\mathrm{Tr}(I-P_d)\\
&= \sigma^2+\frac{p-d}{p}\Vert\beta^\star\Vert_2^2,
\end{align*}
using that $\mathrm{Tr}(I-P_d)=\mathrm{rank}(I-P_d)=p-d$. Besides, if $d\geq p$, $P_d=I_p$, and $\esp R^\star(d)=\sigma^2$.

\subsection{Proof of \Cref{prop:Riskmissing}}

    Using \Cref{lem:decompo_pat_esperance}, we have
    \begin{equation*}
    \esp R_{\rm miss}^\star(d)=\sum_{k=0}^d \P(B=k)\esp R^\star(k),
\end{equation*}
where $B$ is a binomial random variable of parameters $d$ and $\rho$. Using \Cref{prop:RiskcompletDD}, we have 
 \[ 
        \esp \left[ R^\star(k) \right]=\left\{
        \begin{array}{ll}
        \sigma^2+ \frac{p-k}{p}\Vert \beta^\star\Vert_2^2, & \text{when } k< p,\\
        \sigma^2 & \text{when } k \geq p.
        \end{array}
        \right.
    \]
Combining the two previous equalities, we obtain that 
\begin{equation*}
    \esp R_{\rm miss}^\star(d)=  \sigma^2+ \frac{\esp[(p-B)\ind_{B\leq p}]}{p}\Vert \beta^\star \Vert_2^2.
\end{equation*}
In the case where $d\leq p$, $\ind_{B\leq p}=1$ almost surely,  and we obtain
\begin{equation*}
    \esp R_{\rm miss}^\star(d)=  \sigma^2+ \frac{p-\rho d}{p}\Vert \beta^\star \Vert_2^2.
\end{equation*}

\subsection{Proof of \Cref{thm:imputationDD} } 

\subsubsection{Preliminaries}

In the rest of the proof, we denote by $\mathbf{W}=(W_1,\hdots , W_d)^\top\in \mathbb{R}^{d\times p}$ the weight matrix that admits the weight vectors $W_j\sim \mathcal{U}(\mathbb{S}^{p-1})$ for rows.
We call $\Sigma = \mathbb{E}\left[ XX^\top | \mathbf{W}\right]$ the covariance matrix of an input $X\in \mathbb{R}^d$ given the weight matrix $\mathbf{W}$. Recall that the latter, resulting from a random feature model, is such that $X= \mathbf{W}Z$, for $Z\in \mathbb{R}^p$ the corresponding latent vector.

\begin{lemma}\label{lem:DeltaImpFormTrace}
    Under assumptions of \Cref{thm:imputationDD},
    \begin{equation*}
        \Delta_{\rm imp/miss}+\Delta_{\rm miss}= \begin{cases}
\frac{\lambda\left\Vert \beta^{\star}\right\Vert _{2}^{2}}{p}\esp\mathrm{Tr}\left((\Sigma+\lambda I_{d})^{-1}\right) & \text{if }d<p\\
\frac{\lambda\left\Vert \beta^{\star}\right\Vert _{2}^{2}}{p}\esp\mathrm{Tr}\left((\mathbf{W}^{\top}\mathbf{W}+\lambda I_{p})^{-1}\right) & \text{if }d\geq p,
\end{cases}
    \end{equation*}
    with $\lambda=\frac{1-\rho}{\rho}$.
\end{lemma}
\begin{proof}
    One has for $\theta \in \mathbb{R}^p$,
    \[
R(\theta)=\sigma^{2}+\esp\left[(Z^{\top}\beta^{\star}-\theta^{\top}X)^{2}|\mathbf{W}\right].
\]
Using that $X=\mathbf{W}Z$, we have 

\begin{align*}
R(\theta) & =\sigma^{2}+\esp\left[\left(\left(\beta^{\star}-\mathbf{W}^{\top}\theta\right)^{\top}Z\right)^{2}|\mathbf{W}\right]\\
 & =\sigma^{2}+\left\Vert \beta^{\star}-\mathbf{W}^{\top}\theta\right\Vert _{2}^{2},
\end{align*}
by isotropy of $Z$. Since $\theta^\star$ minimizes the risk $R$, $\theta^\star$ minimizes the least-squares criterion above. 
Therefore, in the case where $p>d$ (the ``design" $\mathbf{W}^{\top}$ being long), $\theta^\star$ is unique and given by $\theta^\star = (\mathbf{W}\mathbf{W}^{\top})^{-1}\mathbf{W}\beta^\star$. In the case where $d<p$ (the design matrix $\mathbf{W}^\top$ being fat), there exists an infinite number of minimizers (all are solutions of the system $\beta^\star=\mathbf{W}^\top \theta$), but one can look at the solution of minimal $\ell^2$-norm. Then, KKT conditions provide the particular solution $\theta^\star = \mathbf{W} (\mathbf{W}^\top \mathbf{W})^{-1}\beta^\star$. In both cases, $\theta^\star$ can be written in the following unified way:
\[
\theta^{\star}=\left(\mathbf{W}^{\top}\right)^{\dagger}\beta^{\star}.
\]
Futhermore, 
\begin{align*}
\Sigma & =\esp\left[XX^{\top}|\mathbf{W}\right]=\esp\left[\mathbf{W}Z(\mathbf{W}Z)^{\top}|\mathbf{W}\right]=\esp\left[\mathbf{W}ZZ^{\top}\mathbf{W}^{\top}|\mathbf{W}\right] =\mathbf{W}\mathbf{W}^{\top}.
\end{align*}
Then, using \eqref{eq:BiasRidgeFormClose},
\begin{align*}
\Delta_{\rm imp/miss}+\Delta_{\rm miss} & =\lambda\esp\left\Vert \theta^{\star}\right\Vert _{\Sigma(\Sigma+\lambda I)^{-1}}\\
 & =\lambda\esp\left\Vert \left(\mathbf{W}^{\top}\right)^{\dagger}\beta^{\star}\right\Vert _{\mathbf{W}\mathbf{W}^{\top}(\mathbf{W}\mathbf{W}^{\top}+\lambda I)^{-1}}^{2}\\
 & =\lambda\esp\left\Vert \beta^{\star}\right\Vert _{\mathbf{W^{\dagger}W}\mathbf{W}^{\top}(\mathbf{W}\mathbf{W}^{\top}+\lambda I)^{-1}\left(\mathbf{W}^{\top}\right)^{\dagger}}^{2}\\
 & =\frac{\lambda\left\Vert \beta^{\star}\right\Vert _{2}^{2}}{p}\esp\mathrm{Tr}\left(\mathbf{W^{\dagger}W}\mathbf{W}^{\top}(\mathbf{W}\mathbf{W}^{\top}+\lambda I)^{-1}\left(\mathbf{W}^{\top}\right)^{\dagger}\right),
\end{align*}
using \Cref{lem:vector_random_matrice} remarking that, for all orthogonal matrix $O\in\R^{p\times p}$
\begin{align*}
    O\mathbf{W^{\dagger}W}\mathbf{W}^{\top}(\mathbf{W}\mathbf{W}^{\top}+\lambda I)^{-1}\left(\mathbf{W}^{\top}\right)^{\dagger}O^\top&= O\mathbf{W^{\dagger}W}O^\top O\mathbf{W}^{\top}(\mathbf{W}O^\top O\mathbf{W}^{\top}+\lambda I)^{-1}\left(\mathbf{W}^{\top}\right)^{\dagger}O^\top\\
    &= (\mathbf{W}O^\top)^{\dagger}\mathbf{W}O^\top (\mathbf{W}O^\top)^{\top}(\mathbf{W}O^\top (\mathbf{W}O^\top)^{\top}+\lambda I)^{-1}\left((\mathbf{W}O^\top)^{\top}\right)^{\dagger},
\end{align*}
by orthogonality of $O$ ($O^\top O=I_p$). 
Then, $O\mathbf{W^{\dagger}W}\mathbf{W}^{\top}(\mathbf{W}\mathbf{W}^{\top}+\lambda I)^{-1}\left(\mathbf{W}^{\top}\right)^{\dagger}O^\top$ has the same distribution as  $\mathbf{W^{\dagger}W}\mathbf{W}^{\top}(\mathbf{W}\mathbf{W}^{\top}+\lambda I)^{-1}\left(\mathbf{W}^{\top}\right)^{\dagger}$, since $\mathbf{W}O^\top \overset{dist}{=}\mathbf{W}$. 

Consider the singular value decomposition (SVD) of $\mathbf{W}$,
\[
\mathbf{W}=\sum_{j=1}^{r}\sigma_{j}u_{j}v_{j}^{\top},
\]
where $r=p\wedge d$ is the rank of $\mathbf{W^{\dagger}}$, $(u_{j})$
is an orthonormal basis of $\R^{d},$ and $(v_{j})$ is an orthonormal
basis of $\R^{p}$. The SVD of its pseudo-inverse is therefore
\[
\mathbf{W}^{\dagger}=\sum_{j=1}^{r}\sigma_{j}^{-1}v_{j}u_{j}^{\top}.
\]
Then, we obtain 
\[
\mathbf{W^{\dagger}W}\mathbf{W}^{\top}(\mathbf{W}\mathbf{W}^{\top}+\lambda I)^{-1}\left(\mathbf{W}^{\top}\right)^{\dagger}=\sum_{j=1}^{r}\frac{1}{\lambda+\sigma_{j}^{2}}u_{j}u_{j}^{\top}.
\]
Thus, $$\mathrm{Tr}\left(\mathbf{W^{\dagger}W}\mathbf{W}^{\top}(\mathbf{W}\mathbf{W}^{\top}+\lambda I)^{-1}\left(\mathbf{W}^{\top}\right)^{\dagger}\right)=\sum_{j=1}^{r}\frac{1}{\lambda+\sigma_{j}^{2}}.$$
We recognize $(\lambda+\sigma_{j}^{2})_{j\in[r]}$ as the eigenvalues of $\Sigma+\lambda I_{d}=\mathbf{W}\mathbf{W}^{\top}+\lambda I_{d}$
when $d<p$ and $\text{rank}(\mathbf{W})=d$, or as the eigenvalues of $\mathbf{W}^{\top}\mathbf{W}+\lambda I_{p}$ when
$d\geq p$ and $\text{rank}(\mathbf{W})=p$. 
Hence, 
\[
\Delta_{\rm imp/miss}+\Delta_{\rm miss}=\begin{cases}
\frac{\lambda\left\Vert \beta^{\star}\right\Vert _{2}^{2}}{p}\esp\mathrm{Tr}\left((\Sigma+\lambda I_{d})^{-1}\right) & \text{if }d<p\\
\frac{\lambda\left\Vert \beta^{\star}\right\Vert _{2}^{2}}{p}\esp\mathrm{Tr}\left((\mathbf{W}^{\top}\mathbf{W}+\lambda I_{p})^{-1}\right) & \text{if }d\geq p.
\end{cases}
\] 
\end{proof}

\subsubsection{Proof of \eqref{eq:RimpUpperBound} ($d<p-1$)} 

\textbf{(First step) Decomposition of $R_{\imp}^\star(d)$.}
Note that for $x\geq 0$ (to be chosen later), one has 
\begin{equation*}
    R_{\imp}^\star(d)\leq R_{\imp}(x\theta^\star)\leq R(x\rho\theta^\star)+\rho(1-\rho)\Vert x\theta^\star\Vert_2^2 = R(x\rho\theta^\star)+\rho(1-\rho)x^2\Vert \theta^\star\Vert_2^2,
\end{equation*}
using \Cref{lem_previous_work1}. 
Then, 
\begin{equation*}
    R_{\imp}^\star(d)-R^\star(d)\leq R(x\rho\theta^\star)-R^\star(d)+\rho(1-\rho)x^2\Vert \theta^\star\Vert_2^2.
\end{equation*}
Note that, 
\begin{align*}
    R(x\rho\theta^\star)-R^\star(d)&=\Vert x\rho\theta^\star-\theta^\star\Vert_{\Sigma}^2\\
    &= (1-x\rho)^2\Vert \theta^\star\Vert_{\Sigma}^2. 
\end{align*}
Thus, we have, 
\begin{equation}
    R_{\imp}^\star(d)-R^\star(d)\leq (1-x\rho)^2\Vert \theta^\star\Vert_{\Sigma}^2+x^2\rho(1-\rho)\Vert \theta^\star\Vert_2^2.
\end{equation}

\textbf{(Second step) Calculus of $\esp \Vert \theta^\star\Vert_2^2$.} 
Since $\theta^\star= (\mathbf{W}^\top)^\dag \beta$,
\begin{equation*}
    \esp \Vert \theta^\star\Vert_2^2= \esp \beta^\top(\mathbf{W}^\dag(\mathbf{W^\top})^\dag) \beta.  
\end{equation*}
Again, for an orthonormal matrix $O$, $O\mathbf{W}^\dag(\mathbf{W^\top})^\dag O^\top= (\mathbf{W}O^\top)^\dag((\mathbf{W}O^\top)^\top)^\dag$ has the same law as that of $\mathbf{W}^\dag(\mathbf{W^\top})^\dag$ because $\mathbf{W}O^\top$ has the same law of $\mathbf{W}$. 
Using \Cref{lem:vector_random_matrice}, 
\begin{align*}
    \esp \Vert \theta^\star\Vert_2^2 & = \Vert\beta\Vert_2^2\frac{\esp \mathrm{Tr}\left((\mathbf{W}^\dag)^\top \mathbf{W}^\dag\right)}{p}\\
    & = \Vert\beta\Vert_2^2\frac{\esp \Vert \mathbf{W}^\dag\Vert_{\rm Fr}^2}{p},
\end{align*}
by definition of the Frobenius norm. Then, by \Cref{lem:RMfrobenius}, we obtain 
\begin{equation}\label{eq:norm2boundRFgaussien}
    \esp \Vert \theta^\star\Vert_2^2= \frac{d}{p}\left(1+\frac{d}{p-d-1}\right)\Vert\beta\Vert_2^2. 
\end{equation}

\textbf{(Third step) Calculus of $\esp \Vert \theta^\star\Vert_{\Sigma}^2$.} We have by optimization result, 
\begin{equation*}
    \esp Y^2=  \Vert \theta^\star\Vert_{\Sigma}^2+ R^\star(d).
\end{equation*}
Using that, $\esp Y^2=\sigma^2+\Vert \beta\Vert_2^2$, and taking the expectation, we obtain
\begin{equation*}
    \sigma^2+\Vert \beta\Vert_2^2= \esp \Vert \theta^\star\Vert_{\Sigma}^2+ \esp R^\star(d).
\end{equation*}
Furthermore, by \Cref{prop:RiskcompletDD}, 
\begin{equation*}
        \esp R^\star(d)=\sigma^2+ \frac{p-d}{p}\Vert \beta \Vert_2^2. 
\end{equation*}
Thus, we obtain, 
\begin{equation}
    \esp \Vert \theta^\star\Vert_{\Sigma}^2= \frac{d}{p}\Vert \beta \Vert_2^2. 
\end{equation}

\textbf{(Fourth step) Conclusion.} Putting things together, one gets

\begin{align*}
\esp\left[R_{\imp}^\star(d)-R^\star(d)\right] &\leq (1-x\rho)^2\esp\Vert \theta^\star\Vert_{\Sigma}^2+x^2\rho(1-\rho)\esp\Vert \theta^\star\Vert_2^2\\
    &= \frac{d}{p}\Vert \beta \Vert_2^2\left((1-x\rho)^2 + x^2(1-\rho)\rho\left(1+\frac{d-1}{p-d-1}\right) \right). 
\end{align*}
The bound on the right hand side can be optimized with respect to $x$. 
It corresponds to a strongly convex function of the form $f:x\longmapsto (1-ax)^2+bx^2$. We have $f'(x)=-2a(1-ax)+2bx$, so that the only critical point is $x^\star=\frac{a}{a^2+b}$, leading to  $\min f= f(x^\star)=\frac{b}{a^2+b}$. Therefore,

\begin{align*}
\esp\left[R_{\imp}^\star(d)-R^\star(d)\right] &\leq \frac{d}{p}\Vert \beta \Vert_2^2\frac{(1-\rho)\rho\left(1+\frac{d-1}{p-d-1}\right)}{\rho^2+(1-\rho)\rho\left(1+\frac{d-1}{p-d-1}\right)}\\
&= \frac{d}{p}\Vert \beta \Vert_2^2\frac{(1-\rho)\left(1+\frac{d-1}{p-d-1}\right)}{\rho+(1-\rho)\left(1+\frac{d-1}{p-d-1}\right)}\\
&= \frac{d}{p}\Vert \beta \Vert_2^2\frac{(1-\rho)\left(p-2\right)}{\rho(p-d-1)+(1-\rho)\left(p-2\right)}\\
&= \frac{d}{p}\Vert \beta \Vert_2^2\frac{(1-\rho)\left(p-2\right)}{p-\rho(d-1)-2}\\
&= \frac{d}{p}\Vert \beta \Vert_2^2\left((1-\rho)+ (1-\rho)\frac{\left(p-2\right)-(p-\rho(d-1)-2)}{p-\rho(d-1)-2}\right)\\
&= (1-\rho)\frac{d}{p}\Vert \beta \Vert_2^2\left(1+ \frac{\rho(d-1)}{p-\rho(d-1)-2}\right),
\end{align*}
which leads to the desired result. We obtain also \eqref{eq:DeltaimpmissLowDim} and \eqref{eq:DeltaimpmissLowDim3} using equality obtain in \Cref{prop:Riskmissing}.

\subsubsection{Proof of upper and lower bounds \eqref{eq:DeltaimpmissHighDim1} \eqref{eq:DeltaimpmissHighDim2} ($d\geq p$) }

Using \Cref{lem:DeltaImpFormTrace}, we have 
\begin{equation*}
    \Delta_{\rm imp/miss}+\Delta_{\rm miss}=\frac{\lambda\left\Vert \beta^{\star}\right\Vert _{2}^{2}}{p}\esp\mathrm{Tr}\left((\mathbf{W}^{\top}\mathbf{W}+\lambda I_{p})^{-1}\right)
\end{equation*}
Remark that $\mathbf{W}^{\top}\mathbf{W}=\sum_{j=1}^d W_j W_j^\top$. Furthermore, note that when $W_1\sim\mathcal{U}(\mathbb{S}^{p-1})$, 
for any $1\leq k \leq p$, $W_1 = (W_{11},\hdots, W_{1k}, \hdots , W_{1p})^\top$ has the same distribution as $(W_{11},\hdots, -W_{1k}, \hdots , W_{1p})^\top$. Therefore, for all $1\leq k\neq k' \leq p$, $\mathbb{E}[W_{1k}W_{1k'}] = -\mathbb{E}[W_{1k}W_{1k'}]$, leading to
$\mathbb{E}[W_{1k}W_{1k'}]=0$.
Furthermore $\mathbb{E}[W_{11}^2+\hdots + W_{1p}^2]=\mathbb{E}[W_{11}^2] +\hdots + \mathbb{E}[W_{1p}^2]=1$, so that by exchangeability, for all $1\leq k \leq p$, $\mathbb{E}[W_{1k}^2] =1/p$ and finally $\esp W_1W_1^\top=\frac{1}{p} I_p$.

Applying \Cref{prop:exchangeabilityResults}, we obtain 
\begin{equation*}
    \frac{\lambda\left\Vert \beta^{\star}\right\Vert _{2}^{2}}{p} \frac{p}{d/p+\lambda}\leq \Delta_{\rm imp/miss}+\Delta_{\rm miss}\leq \frac{\lambda\left\Vert \beta^{\star}\right\Vert _{2}^{2}}{p}(1+1/\lambda)\frac{p}{d/p+\lambda},
\end{equation*}
and
\begin{equation*}
    \left\Vert \beta^{\star}\right\Vert _{2}^{2}(1-\rho) \frac{p}{\rho d+(1-\rho)p}\leq \Delta_{\rm imp/miss}+\Delta_{\rm miss}\leq \left\Vert \beta^{\star}\right\Vert _{2}^{2} \frac{p}{\rho d+(1-\rho)p}.
\end{equation*}

\section{Proof of \Cref{sec:ExtensionRF}}\label{app:Extension}

\subsection{Proof of \Cref{thm:DeltaExtension}}
Start by writing 
\begin{align*}
    \Delta_{\rm imp}^{(\infty)}&= \esp R_{\rm imp}^\star(d)-R^\star(\infty)\\
    &= \Delta_{\rm imp/miss}+\Delta_{\rm miss}+\esp R^\star(d)-R^\star(\infty)\\
    &= \esp \inf_{\theta\in\R^d}\left\{R(\theta)-R^\star(\infty)+\frac{1-\rho}{\rho}\Vert\theta\Vert_{{\rm diag}(\Sigma)}^2\right\},
\end{align*}
using \eqref{eq:BiasRidgeForm}. Considering \Cref{ass:nongaussianRF}, $\mathrm{diag}(\Sigma)\preceq L^2I_d$, that leads to 
\begin{equation*}
    \Delta_{\rm imp}^{(\infty)}\leq \esp \inf_{\theta\in\R^d}\left\{R(\theta)-R^\star(\infty)+L^2\frac{1-\rho}{\rho}\Vert\theta\Vert_{2}^2\right\}.  
\end{equation*}
Fixing $\lambda= L^2\frac{1-\rho}{\rho}$, we aim at providing an upper bound for 
\begin{equation*}
    \Delta_{\lambda}:= \esp \inf_{\theta\in\R^d}\left\{R(\theta)-R^\star(\infty)+\lambda\Vert\theta\Vert_{2}^2\right\}.  
\end{equation*}
Let $f\in\F_{\nu}^{(\infty)}$ with one representative $\alpha\in L^2(\nu)$.
We set $\theta^{(\alpha)}\in\R^{d}$ such that for all $j\in[d]$, $\theta_{j}^{(\alpha)}=\frac{1}{d}\alpha(W_{j})$.
We have,

\begin{align*}
\Delta_{\lambda}  +R^\star(\infty) & =\esp\inf_{\theta\in\R^{d}}\left\{ R\left(\theta\right)+\lambda\esp\left[\left\Vert \theta\right\Vert _{2}^{2}|\mathbf{W}\right]\right\} \\
 & \leq\esp\left[R\left(\theta^{(\alpha)}\right)+\lambda\esp\left[\left\Vert \text{\ensuremath{\theta}}^{(\alpha)}\right\Vert _{2}^{2}|\mathbf{W}\right]\right]\\
 & =\esp R\left(\text{\ensuremath{\theta}}^{(\alpha)}\right)+\lambda\esp\left[\left\Vert \text{\ensuremath{\theta}}^{(\alpha)}\right\Vert _{2}^{2}\right]
\end{align*}
\textbf{First term.} Remark that by definition of random features \eqref{eq:RFmodel},  $X^\top\text{\ensuremath{\theta}}^{(\alpha)}=\sum_{j=1}^{d}\theta_{j}^{(\alpha)}\psi (Z,W_j)=\frac{1}{d}\sum_{j=1}^{d}\alpha(W_{j})\psi (Z,W_j)$.
In consequence, $\esp\left[X^\top\text{\ensuremath{\theta}}^{(\alpha)}|Z\right]=\int\alpha(W)\psi (Z,W)d\nu(W)=f(Z)$.
Then 
\begin{align*}
\esp R\left(\text{\ensuremath{\theta}}^{(\alpha)}\right) & =\esp\left[\esp\text{\ensuremath{\left[\left(X^\top\text{\ensuremath{\theta}}^{(\alpha)}-Y\right)^{2}|\mathbf{W}\right]}}\right]\\
 & =\esp\left[\esp\text{\ensuremath{\left[\left(X^\top\text{\ensuremath{\theta}}^{(\alpha)}-Y\right)^{2}|Z\right]}}\right] & {\text{using Fubini's theorem}}\\
 & =\esp\left[\esp\text{\ensuremath{\left[\left(X^\top\text{\ensuremath{\theta}}^{(\alpha)}-f(Z)+f(Z)-Y\right)^{2}|Z\right]}}\right]\\
 & =\esp\left[\esp\text{\ensuremath{\left[\left(X^\top\text{\ensuremath{\theta}}^{(\alpha)}-f(Z)\right)^{2}+\left(f(Z)-Y\right)^{2}|Z\right]}}\right] & \text{using }\esp\left[X^\top\text{\ensuremath{\theta}}^{(\alpha)}|Z\right]=f(Z)\\
 & =\esp\left[\mathbb{V}\left[X^\top\text{\ensuremath{\theta}}^{(\alpha)}|Z\right]\right]+R(f).
\end{align*}
Then, 
\begin{align*}
\mathbb{V}\left[X^\top\text{\ensuremath{\theta}}^{(\alpha)}|Z\right] & =\mathbb{V}\left[\frac{1}{d}\sum_{j=1}^{d}\alpha(W_{j})\psi (Z,W_j)|Z\right]&\\
 & =\frac{1}{d}\mathbb{V}_{\nu}\left[\alpha(W)\psi (Z,W)|Z\right] & \phantom{aaaaaaaaaaaa}(W_j) \text{  being i.i.d}\\
 & \leq\frac{1}{d}\esp_{\nu}\left[\alpha(W)^{2}\psi (Z,W)^{2}|Z\right]&.
\end{align*}
Then using Fubini's theorem, 
\begin{align*}
\mathbb{V}\left[X^\top\text{\ensuremath{\theta}}^{(\alpha)}\right] & \leq  \frac{1}{d}\esp\left[\esp\left[\alpha(W)^{2}\psi (Z,W)^{2}|W\right]\right]=\frac{1}{d}\esp\left[\alpha(W)^{2}\esp\left[\psi (Z,W)^{2}|W\right]\right].
\end{align*}
Under \Cref{ass:nongaussianRF},
\begin{align*}
\mathbb{V}\left[X^\top\text{\ensuremath{\theta}}^{(\alpha)}\right] & \leq  \frac{L^{2}}{d}\esp\left[\alpha(W)^{2}\right].
\end{align*}
Thus, 
\[
\esp R\left(\text{\ensuremath{\theta}}^{(\alpha)}\right)\leq\frac{L^{2}}{d}\esp_{\nu}\left[\alpha(W)^{2}\right]+R(f).
\]
\textbf{Second term.} 
\begin{align*}
\esp\left[\left\Vert \text{\ensuremath{\theta}}^{(\alpha)}\right\Vert _{2}^{2}\right] & =\esp\left[\sum_{j=1}^{d}\theta_{j}^{2}\right]\\
 & =\frac{1}{d}\esp\left[\frac{1}{d}\sum_{j=1}^{d}\alpha(W_{j})^{2}\right]\\
 & =\frac{1}{d}\esp\left[\alpha(W)^{2}\right],
\end{align*}
using that $(W_j)_j$ are i.i.d..

\textbf{Conclusion.} Combining these two terms, we have 
\[
\Delta_{\lambda}  +R^\star(\infty)\leq R(f)+\frac{\lambda+L^{2}}{d}\esp_{\nu}\left[\alpha(W)^{2}\right].
\]
This result is valid for any $f$ and $\alpha$, thus 

\[
\Delta_{\rm imp}^{(\infty)}+R^\star(\infty)\leq \Delta_{\lambda}  +R^\star(\infty)\leq\inf_{f\in\F^{(\nu)}}\left\{ R(f)+\frac{\lambda+L^{2}}{d}\left\Vert f\right\Vert _{\nu}^{2}\right\} .
\]

\section{Proof of error bounds for SGD estimators }\label{app:SGDbound}
\subsection{Proof of equation~\eqref{eq:GenGaussianLD} of \Cref{thm:sgd_rf_gaussian}}
In this section, we apply results from the SGD literature, in particular, \citet[][Theorem 1]{bach2013non}, to our framework.
\begin{theorem} \label{thm:bachmoulines}
In the framework of \Cref{sec:DD}, for $\gamma= \frac{1}{6 d}$, and when $d<n$, we have 
\begin{equation}
    R_{\imp}(\bar \theta)-R_{\imp}^\star(d) \lesssim \frac{d}{n}\Vert\theta_{\rm imp}^\star\Vert_2^2+ \frac{d}{n}\sigma_{\rm imp}^2,
\end{equation}
with $\sigma_{\rm imp}^2:= \sigma^2+ \rho^{-1}(R_{\imp}^\star(d)-R^\star(d)+\left\Vert \theta^\star\right\Vert _{\Sigma}^{2})$.
    
\end{theorem}
\begin{proof}
    The proof of this theorem consists of verifying that assumptions of \citet[][Theorem 1]{bach2013non} hold in our case.  Assumptions (A1-5) are easily satisfied. Let us show that $
\esp\left[\tilde X\tilde X^{\top}\left\Vert \tilde X\right\Vert _{2}^{2} |\mathbf{W}\right]\preceq R^{2}\Sigma_{\imp}
$. Indeed,
\[
\esp\left[\tilde X\tilde X^{\top}\left\Vert \tilde X\right\Vert _{2}^{2}|\mathbf{W}\right]\preceq\esp\left[\tilde X\tilde X^{\top}\left\Vert X\right\Vert _{2}^{2}|\mathbf{W}\right],
\]

using that $\left\Vert \tilde X\right\Vert _{2}^{2}\leq\left\Vert X\right\Vert _{2}^{2}$, and $0\preceq \tilde X\tilde X^{\top}$.
Then, 
\begin{align*}
\esp\left[\tilde X\tilde X^{\top}\left\Vert X\right\Vert _{2}^{2}|\mathbf{W}\right] & =\esp\esp\left[\tilde X\tilde X^{\top}\left\Vert X\right\Vert _{2}^{2}|P,\mathbf{W}\right]\\
 & =\esp\esp\left[PP^{\top}\odot XX^{\top}\left\Vert X\right\Vert _{2}^{2}|P,\mathbf{W}\right]\\
 & =\esp\left[\Sigma_{P}\odot XX^{\top}\left\Vert X\right\Vert _{2}^{2}|\mathbf{W}\right]\\
 & =\Sigma_{P}\odot\left(\esp\left[XX^{\top}\left\Vert X\right\Vert _{2}^{2}|\mathbf{W}\right]\right).
\end{align*}

$X$ is Gaussian vector, thus $\esp\left[XX^{\top}\left\Vert X\right\Vert _{2}^{2}\right]\preceq R^{2}\Sigma$ with $R^2=3 d$,
and \cref{lem:hadamard_monotonicity} lead to 
\[
\esp\left[\tilde X\tilde X^{\top}\left\Vert \tilde X\right\Vert _{2}^{2}\right]\preceq R^{2}\Sigma_{P}\odot\Sigma=R^{2}\Sigma_{\imp}.
\]

 Define $\epsilon_{\imp}=Y-\tilde X^{\top}\theta_{\imp}^{\star}=X^{\top}\theta^{\star}+\epsilon-\tilde X^{\top}\theta_{\imp}^{\star}$
. First, we have $\epsilon_{\imp}^{2}\leq 3\left(\epsilon^{2}+\left(\tilde X^{\top}\theta_{\imp}^{\star}\right)^{2}+\left(X^{\top}\theta^{\star}\right)^{2}\right)$,
then 
\begin{align}
\label{eq:3terms}
\esp\left[\epsilon_{\imp}^{2} \tilde X\tilde X^\top \right] & \preceq 3 \esp\left[\epsilon^2 \tilde X\tilde X^\top\right]+3 \esp\left[\left(\tilde X^{\top}\theta_{\imp}^{\star}\right)^{2} \tilde X\tilde X^\top\right]+3\esp\left[\left(\tilde X^{\top}\theta^{\star}\right)^{2} \tilde X\tilde X^\top\right].
\end{align}
Using that $\epsilon$ is an independent noise,  $\esp\left[\epsilon^2 \tilde X\tilde X^\top\right]=\sigma^2\Sigma_{\rm imp}$. Let $u,v$ in $\R^d$, note that 
\begin{align*}
v^{\top}\esp\left[\left(u^{\top}\tilde{X}\right)^{2}\tilde{X}\tilde{X}^{\top}\right]v & =\esp\left[\left(u^{\top}\tilde{X}\right)^{2}\left(v^{\top}\tilde{X}\right)^{2}\right]\\
 & \leq\esp\left[\left(u^{\top}X\right)^{2}\left(v^{\top}\tilde{X}\right)^{2}\right]\\
 & \leq\rho\esp\left[\left(u^{\top}X\right)^{2}\left(v^{\top}X\right)^{2}\right]\\
 & \leq\rho\sqrt{\esp\left[\left(u^{\top}X\right)^{4}\right]\esp\left[\left(v^{\top}X\right)^{4}\right]},
\end{align*}
using the Cauchy-Schwarz inequality. Then, by the kurtosis boundedness of Gaussian vectors, we have
$\esp\left[\left(u^{\top}X\right)^{4}\right]\leq3\left\Vert u\right\Vert _{\Sigma}^{4}$
and $\esp\left[\left(v^{\top}X\right)^{4}\right]\leq3\left\Vert v\right\Vert _{\Sigma}^{4}$.
Then, 

\begin{align*}
v^{\top}\esp\left[\left(u^{\top}\tilde{X}\right)^{2}\tilde{X}\tilde{X}^{\top}\right]v & =3\rho\left\Vert u\right\Vert _{\Sigma}^{2}\left\Vert v\right\Vert _{\Sigma}^{2}\\
 & \leq3\rho^{-1}\left\Vert u\right\Vert _{\Sigma}^{2}\left\Vert v\right\Vert _{\Sigma_{{\rm imp}}}^{2}.
\end{align*} 

This shows that
\begin{equation}
    \esp\left[\left(u^{\top}\tilde{X}\right)^{2}\tilde{X}\tilde{X}^{\top}\right]\preceq 3\rho^{-1}\left\Vert u\right\Vert _{\Sigma}^{2}\Sigma_{{\rm imp}}.
\end{equation}

Using similar arguments,
\begin{equation}
    \esp\left[\left(u^{\top}X\right)^{2}\tilde{X}\tilde{X}^{\top}\right]\preceq 3\rho^{-1}\left\Vert u\right\Vert _{\Sigma}^{2}\Sigma_{{\rm imp}}.
\end{equation}
These two above equations can be used when $u$ is equal to $\theta_{\imp}^\star$ and $\theta^\star$, to transform \eqref{eq:3terms} into
\begin{align*}
\esp\left[\epsilon_{\imp}^{2} \tilde X\tilde X^\top \right] & \preceq (3\sigma^2+9\rho^{-1}\left\Vert \theta_{\imp}^\star\right\Vert _{\Sigma}^{2}+9\rho^{-1}\left\Vert \theta^\star\right\Vert _{\Sigma}^{2} )\Sigma_{{\rm imp}}.
\end{align*}
Remarking, that $\left\Vert \theta_{\imp}^\star\right\Vert _{\Sigma}^{2}\leq 2 \left\Vert \theta_{\imp}^\star-\theta^\star\right\Vert _{\Sigma}^{2}+2\left\Vert \theta^\star\right\Vert _{\Sigma}^{2}=2(R_{\imp}^\star(d)-R^\star(d)+\left\Vert \theta^\star\right\Vert _{\Sigma}^{2})$, we get
\begin{equation*}
    3\sigma^2+9\rho^{-1}\left\Vert \theta_{\imp}^\star\right\Vert _{\Sigma}^{2}+9\rho^{-1}\left\Vert \theta^\star\right\Vert _{\Sigma}^{2} \lesssim \sigma^2+ \rho^{-1}(R_{\imp}^\star(d)-R^\star(d)+\left\Vert \theta^\star\right\Vert _{\Sigma}^{2}),
\end{equation*}
leading to the desired results.
\end{proof}
\begin{lemma}
    Under assumptions of \Cref{thm:sgd_rf_gaussian}.  The norm of $\theta_{\imp}^\star$, the best predictor working with imputed by $0$ inputs, satisfies 
    \begin{equation*}
        \esp\Vert\theta_{\imp}^\star\Vert^2\leq
        \left\{
        \begin{array}{ll}
         \frac{d}{p\rho}\frac{p-2}{p-d-1}\Vert \beta \Vert_2^2 & \text{when } d< p-1,\\
        \frac{p}{d\rho^2(1-\rho)}\Vert \beta \Vert_2^2, & \text{when } d \geq p-1.
        \end{array}
        \right.
    \end{equation*}
\end{lemma}
\begin{proof}
Let's begin by,
\begin{equation}
     \rho(1-\rho)\Vert\theta_{\rm imp}^\star\Vert_2^2\leq R(\rho\theta_{\rm imp}^\star)-R(\rho\theta^\star)+\rho(1-\rho)\Vert\theta_{\rm imp}^\star\Vert_2^2.
\end{equation}
because $R(\theta^\star)\leq R(\rho\theta_{\rm imp}^\star)$. Using, \Cref{lem_previous_work1}, we obtain 
\begin{equation}\label{eq:normvsdelta}
    \rho(1-\rho)\esp\Vert\theta_{\rm imp}^\star\Vert_2^2\leq \Delta_{\rm miss}+\Delta_{\rm imp/miss}
\end{equation}
\textbf{First case: $d<p-1$. }
In this, case 
\begin{equation*}
    \Delta_{\rm miss}+\Delta_{\rm imp/miss}\leq (1-\rho)\frac{d}{p}\Vert \beta \Vert_2^2\left(1+ \frac{\rho(d-1)}{p-\rho(d-1)-2}\right).
\end{equation*}
We obtain using \eqref{eq:normvsdelta},
\begin{equation*}
    \esp\Vert\theta_{\rm imp}^\star\Vert_2^2\leq \frac{d}{\rho p}\Vert \beta \Vert_2^2\left(1+ \frac{\rho(d-1)}{p-\rho(d-1)-2}\right)\leq \frac{p-2}{\rho p}\Vert \beta \Vert_2^2 \frac{\rho(d-1)}{p-\rho(d-1)-2} .
\end{equation*}
\textbf{Second case: $d>p-1$. }
In this case, 
\begin{equation*}
    \Delta_{\rm miss}+\Delta_{\rm imp/miss}\leq \frac{p}{\rho d+(1-\rho)p}\Vert \beta^\star \Vert_2^2\leq \frac{p}{\rho d}\Vert \beta^\star \Vert_2^2.
\end{equation*}
We obtain using \eqref{eq:normvsdelta},
\begin{equation*}
    \esp\Vert\theta_{\rm imp}^\star\Vert_2^2\leq \frac{p}{\rho^2(1-\rho) d}\Vert \beta^\star \Vert_2^2 .
\end{equation*}
\end{proof}

\begin{proof}[Proof of \eqref{eq:GenGaussianLD}]
    \begin{align*}
            R_{\imp}(\bar \theta)-R^\star(d) &= R_{\imp}(\bar \theta)-R_{\imp}^\star(d)+R_{\imp}^\star(d)-R^\star(d)\\
    \end{align*}
Using \Cref{thm:bachmoulines} to bound the first term, we find
\begin{equation*}
    R_{\imp}(\bar \theta)-R^\star(d) \lesssim \left(1+\frac{d}{\rho n}\right)(R_{\imp}^\star(d)-R^\star(d))+\frac{d}{n}\Vert\theta_{\rm imp}^\star\Vert_2^2+ \frac{d}{n}(\sigma^2+\left\Vert \theta^\star\right\Vert _{\Sigma}^{2}). 
\end{equation*}
Thus, taking the expectation, 
\begin{equation*}
    \esp[R_{\imp}(\bar \theta)-R^\star(d)] \lesssim \left(1+\frac{d}{\rho n}\right)(\Delta_{\rm imp/miss}+\Delta_{\rm miss})+\frac{d}{n}\esp\Vert\theta_{\rm imp}^\star\Vert_2^2+ \frac{d}{n}(\sigma^2+\esp\left\Vert \theta^\star\right\Vert _{\Sigma}^{2}). 
\end{equation*}
Note that $\esp\Vert\theta_{\rm imp}^\star\Vert_2^2\leq \frac{1}{(1-\rho)\rho}(\Delta_{\rm imp/miss}+\Delta_{\rm miss})$ (using \eqref{eq:normvsdelta}) and $\esp\left\Vert \theta^\star\right\Vert _{\Sigma}^{2}=\frac{d}{p}\left\Vert \beta^\star\right\Vert _{2}^{2}$ (using \Cref{prop:RiskcompletDD}). Thus, 
\begin{align*}
    \esp[R_{\imp}(\bar \theta)-R^\star(d)] &\lesssim \left(1+\frac{d}{\rho n}+ \frac{d}{(1-\rho)\rho n}\right)(\Delta_{\rm imp/miss}+\Delta_{\rm miss}) +\frac{d}{n}\left(\sigma^2+\frac{d}{p}\left\Vert \beta^\star\right\Vert _{2}^{2}\right)\\
    &\lesssim \left(1+ \frac{d}{(1-\rho)\rho n}\right)(\Delta_{\rm imp/miss}+\Delta_{\rm miss}) +\frac{d}{n}\left(\sigma^2+\frac{d}{p}\left\Vert \beta^\star\right\Vert _{2}^{2}\right).
\end{align*}
Thus,
\begin{align*}
    \esp[R_{\imp}(\bar \theta)-R_{\rm miss}^\star(d)] &\lesssim \Delta_{\rm imp/miss}+\frac{d}{\rho n}+ \frac{d}{(1-\rho)\rho n}(\Delta_{\rm imp/miss}+\Delta_{\rm miss}) +\frac{d}{n}\left(\sigma^2+\frac{d}{p}\left\Vert \beta^\star\right\Vert _{2}^{2}\right)\\
    &\lesssim (1-\rho)\frac{d}{p}\Vert \beta \Vert_2^2 \frac{\rho(d-1)}{p-\rho(d-1)-2} +\frac{d}{(1-\rho)\rho n}\frac{d}{p}\Vert \beta \Vert_2^2\frac{(1-\rho)\left(p-2\right)}{p-\rho(d-1)-2} +\frac{d}{n}\left(\sigma^2+\frac{d}{p}\left\Vert \beta^\star\right\Vert _{2}^{2}\right)\\
    &\lesssim (1-\rho)\frac{d}{p}\Vert \beta \Vert_2^2 \frac{\rho(d-1)}{p-\rho(d-1)-2} +\frac{d}{(1-\rho)\rho n}\frac{d}{p}\Vert \beta \Vert_2^2\frac{(1-\rho)\left(p-2\right)}{p-\rho(d-1)-2} +\frac{d}{n}\sigma^2\\ 
    &\lesssim \frac{d^2\Vert \beta \Vert_2^2}{p(p-\rho(d-1)-2)}\left((1-\rho)\rho +\frac{p-2}{\rho n}\right) +\frac{d}{n}\sigma^2.
\end{align*}

\end{proof}

\subsection{Proof of Equation~\eqref{eq:GenGaussianHD} in \Cref{thm:sgd_rf_gaussian} and proof of \Cref{thm:genExtension}}
Let's start with a result of \citet{ayme2023naive} for the deterministic case (without random features). 
\begin{assumption}\label{ass:SGD}
There exist $\sigma>0$ and $R>0$ such that $\esp[XX^{\top}\left\Vert X\right\Vert _{2}^{2}]\preceq R^2\Sigma$ and $\esp [\epsilon^{2}\left\Vert X\right\Vert _{2}^{2}]\leq\sigma^{2}R^2$, where $\epsilon= Y-X^\top\theta^\star$. 
\end{assumption}

\begin{theorem}\label{thm:SGD_bound}
\citep{ayme2023naive}
Under \Cref{ass:SGD}, choosing a constant learning rate $\gamma = \frac{1}{\kappa\mathrm{Tr}(\Sigma)\sqrt{n}}$ leads to 
\begin{align*}
     \esp\left[R_{\imp}\left(\bar{\theta}_{\imp}\right)\right]-R(\theta^\star) \lesssim  \frac{R^2}{\sqrt{n}}\left\Vert \theta^{\star}_{\imp}\right\Vert _{2}^{2} +\frac{\sigma^2+\left\Vert \theta^{\star}\right\Vert _{\Sigma}^{2}}{\sqrt{n}} + R_{\imp}^\star(d)-R^\star(d),
\end{align*}
where $\theta^\star$ (resp. $\theta^{\star}_{\imp}$) is the best linear predictor for complete (resp. with imputed missing values) case. 
\end{theorem}

\begin{proof}[Proof of \eqref{eq:GenGaussianHD} of \Cref{thm:sgd_rf_gaussian}]
$X$ is Gaussian vector, thus $\esp\left[XX^{\top}\left\Vert X\right\Vert _{2}^{2}\right]\preceq R^{2}\Sigma$ with $R^2=2 d$. Furthermore,  in the cass $p\leq d$ $X^\top \theta^\star=Z^\top\beta^\star$ and $\epsilon=Y-X^\top$ thus the noise and $X$ are independent and $\esp [\epsilon^{2}\left\Vert X\right\Vert _{2}^{2}]\leq\sigma^{2}\mathrm{Tr}(\Sigma)$. Then, \Cref{ass:SGD} is satisfied with $\kappa=3$. Taking the expectation in \Cref{thm:SGD_bound}, we found 
\begin{align*}
     \esp\left[R_{\imp}\left(\bar{\theta}_{\imp}\right)\right]-\esp R(\theta^\star) \lesssim  \frac{\mathrm{Tr}(\Sigma)}{\sqrt{n}}\esp\left\Vert \theta^{\star}_{\imp}\right\Vert _{2}^{2} +\frac{\sigma^2+\esp\left\Vert \theta^{\star}\right\Vert _{\Sigma}^{2}}{\sqrt{n}} +\Delta_{\rm imp/miss}+\Delta_{\rm miss},
\end{align*}
Let's recall that, $ R(\theta^\star)=\sigma^2$ almost-surely, $\left\Vert \theta^{\star}\right\Vert _{\Sigma}^{2}=\Vert\beta\Vert_2^2$ almost surely and $\mathrm{Tr}(\Sigma)=d$. Then, 
\begin{align*}
     \esp\left[R_{\imp}\left(\bar{\theta}_{\imp}\right)\right]-\sigma^2 \lesssim  \frac{d}{\sqrt{n}}\esp\left\Vert \theta^{\star}_{\imp}\right\Vert _{2}^{2} +\frac{\sigma^2+\Vert\beta\Vert_2^2}{\sqrt{n}} +\Delta_{\rm imp/miss}+\Delta_{\rm miss}.
\end{align*}
Then applying $\eqref{eq:normvsdelta}$,
\begin{align*}
     \esp\left[R_{\imp}\left(\bar{\theta}_{\imp}\right)\right]-\sigma^2 \lesssim \left(1+\frac{d}{\rho(1-\rho)\sqrt{n}}\right)(\Delta_{\rm imp/miss}+\Delta_{\rm miss}) +\frac{\sigma^2+\Vert\beta\Vert_2^2}{\sqrt{n}}. 
\end{align*}
Applying \eqref{eq:DeltaimpmissHighDim2}, we finally found,
\begin{align*}
     \esp\left[R_{\imp}\left(\bar{\theta}_{\imp}\right)\right]-\sigma^2 &\lesssim \left(1+\frac{d}{\rho(1-\rho)\sqrt{n}}\right) \frac{p\left\Vert \beta^{\star}\right\Vert _{2}^{2}}{\rho d+(1-\rho)p} +\frac{\sigma^2+\Vert\beta\Vert_2^2}{\sqrt{n}}\\ 
     &\lesssim \left(1+\frac{d}{\rho(1-\rho)\sqrt{n}}\right) \frac{p\left\Vert \beta^{\star}\right\Vert _{2}^{2}}{\rho d+(1-\rho)p} +\frac{\sigma^2}{\sqrt{n}}.
\end{align*}
    
\end{proof}
\begin{proof}[Proof of \Cref{thm:genExtension}]
Under \Cref{ass:kurtosis}, $\Vert X\Vert_2^2\leq \kappa L^2d$ almost surely, then $\esp[XX^{\top}\left\Vert X\right\Vert _{2}^{2}]\preceq \kappa L^2 d\Sigma$. And, 
\[
\esp[\epsilon^{2}\left\Vert X\right\Vert _{2}^{2}]\leq\esp[\epsilon^{2}]\kappa L^{2}d=R^\star(\infty)\kappa L^{2}d.
\]
Thus we can applied \Cref{thm:SGD_bound}, that gives us, 
\begin{align*}
     \esp\left[R_{\imp}\left(\bar{\theta}_{\imp}\right)\right]-R(\theta^\star) \lesssim  \frac{\kappa L^2 d}{\sqrt{n}}\left\Vert \theta^{\star}_{\imp}\right\Vert _{2}^{2} +\frac{R^\star(\infty)L²\kappa+\left\Vert \theta^{\star}\right\Vert _{\Sigma}^{2}}{\sqrt{n}} + R_{\imp}^\star(d)-R^\star(d).
\end{align*}
Note that, 
\begin{align*}
    \Delta_{\rm imp}^{(\infty)}&= \esp R_{\rm imp}^\star(d)-R^\star(\infty)\\
    &= \Delta_{\rm imp/miss}+\Delta_{\rm miss}+\esp R^\star(d)-R^\star(\infty). 
\end{align*}
Thus taking the expectation, 
\begin{align*}
     \esp\left[R_{\imp}\left(\bar{\theta}_{\imp}\right)\right]-R^\star(\infty) \lesssim  \frac{\kappa L^2 d}{\sqrt{n}}\esp\left\Vert \theta^{\star}_{\imp}\right\Vert _{2}^{2} +\frac{R^\star(\infty)\kappa L^2+\left\Vert \theta^{\star}\right\Vert _{\Sigma}^{2}}{\sqrt{n}} + \Delta_{\rm imp}^{(\infty)}.
\end{align*}
Under \Cref{ass:momentinf}, $l^2I\preceq \mathrm{diag}(\Sigma)$,
\begin{align*}
        \ell^2\rho(1-\rho)\esp\left\Vert \theta^{\star}_{\imp}\right\Vert _{2}^{2}&\leq \rho(1-\rho) \esp\left\Vert \theta^{\star}_{\imp}\right\Vert _{\mathrm{diag}(\Sigma)}^{2}\\
        &\leq \Delta_{\rm imp/miss}+\Delta_{\rm miss}\\
        &\leq \Delta_{\rm imp}^{(\infty)}.
\end{align*}
Then, 
\begin{align*}
     \esp\left[R_{\imp}\left(\bar{\theta}_{\imp}\right)\right]-R^\star(\infty) &\lesssim  \frac{\kappa L^2 d}{\ell^2\rho(1-\rho)\sqrt{n}}\Delta_{\rm imp}^{(\infty)} +\frac{R^\star(\infty)\kappa+\left\Vert \theta^{\star}\right\Vert _{\Sigma}^{2}}{\sqrt{n}} + \Delta_{\rm imp}^{(\infty)}\\
     &\lesssim  \left(1+\frac{\kappa L^2 d}{\ell^2\rho(1-\rho)\sqrt{n}}\right)\Delta_{\rm imp}^{(\infty)} +\frac{R^\star(\infty)\kappa L^2+\left\Vert \theta^{\star}\right\Vert _{\Sigma}^{2}}{\sqrt{n}}.   
\end{align*}
Recall that using \Cref{thm:DeltaExtension}, 
\[
\Delta_{\rm imp}^{(\infty)}\leq\frac{\lambda_{\rm imp}}{d}\left\Vert f^\star\right\Vert _{\nu}^{2}= \frac{L^2}{\rho d}\left\Vert f^\star\right\Vert _{\nu}^{2},
\]
and $\left\Vert \theta^{\star}\right\Vert _{\Sigma}^{2}\leq \esp Y^2$ almost-surely. Thus
\begin{align*}
     \esp\left[R_{\imp}\left(\bar{\theta}_{\imp}\right)\right]-R^\star(\infty) 
     &\lesssim  \left(1+\frac{\kappa L^2 d}{\ell^2\rho(1-\rho)\sqrt{n}}\right)\frac{l^2}{\rho d}\left\Vert f^\star\right\Vert _{\nu}^{2} +\frac{R^\star(\infty)\kappa+ \esp Y^2}{\sqrt{n}}\\
     &\lesssim \left(1+\frac{\kappa L^2 d}{\ell^2\rho(1-\rho)\sqrt{n}}\right)\frac{L^2}{\rho d}\left\Vert f^\star\right\Vert _{\nu}^{2} +(1+\kappa L^2)\frac{ \esp Y^2}{\sqrt{n}}.
\end{align*}

\end{proof}
\section{Proof of Theorem \ref{thm:convergenceMNAR} (under MNAR assumption)}

\textbf{First step (bias-variance)} 
We denote by $\mathbf{W}'$ the matrix of $\underline{W}_j'$
Let $\theta\in\R^{p}$, 
\begin{align*}
R_{{\rm imp}}(\theta) & =\esp_{Z}\esp\left[\left(Y-\tilde{X}^{\top}\theta\right)^{2}|Z,\mathbf{W},\mathbf{W}'\right]\\
 & =\esp_{Z}\esp\left[\left(Y-\esp\left[\tilde{X}^{\top}\theta|Z,\mathbf{W},\mathbf{W}'\right]\right)^{2}|Z,\mathbf{W},\mathbf{W}'\right]+\esp_{Z}\mathbb{V}\left[\tilde{X}^{\top}\theta|Z,\mathbf{W},\mathbf{W}'\right],
\end{align*}
using bias-variance decomposition. Futhermore, 
\[
\esp\left[\tilde{X}^{\top}\theta|Z,\mathbf{W},\mathbf{W}'\right]=\sum_{j=1}^{d}\theta_{j}\esp\left[\tilde{X}_{j}|Z,\mathbf{W},\mathbf{W}'\right]=\sum_{j=1}^{d}\theta_{j}\phi(Z,\underline{W}_{j}')\psi(Z,W_{j})
\]
and 
\begin{align*}
\mathbb{V}\left[\tilde{X}^{\top}\theta|Z,\mathbf{W},\mathbf{W}'\right] & =\sum_{j=1}^{d}\theta_{j}^{2}\mathbb{V}\left[\tilde{X}_{j}|Z,W_{j},\underline{W}_{j}'\right]\\
 & =\sum_{j=1}^{d}\theta_{j}^{2}\phi(Z,\underline{W}_{j}')(1-\phi(Z,\underline{W}_{j}'))\psi(Z,W_{j})^{2}.
\end{align*}
Let $\alpha\in L^{2}(\mu\otimes\nu)$, and define $\theta^{(d)}\in\R^{d}$
such that $\theta_{j}^{(d)}=\alpha(W_{j},\underline{W}_{j}')/d$. We have 
\begin{align*}
R_{{\rm imp}}^{\star}(d) & \leq R_{{\rm imp}}(\theta^{(d)})\\
 & =\esp_{Z}\left[\left(Y-\frac{1}{d}\sum_{j=1}^{d}\alpha(W_{j},\underline{W}_{j}')\phi(Z,\underline{W}_{j}')\psi(Z,W_{j})\right)^{2}\right]+\esp_{Z}\left[\frac{1}{d^{2}}\sum_{j=1}^{d}\alpha(W_{j},\underline{W}_{j}')^{2}\phi(Z,\underline{W}_{j}')(1-\phi(Z,\underline{W}_{j}'))\psi(Z,W_{j})^{2}\right].
\end{align*}
\textbf{Convergence of variance term} 
Using that $\phi(Z,\underline{W}_{j}')(1-\phi(Z,\underline{W}_{j}'))\leq1$
almost-surely, we have 
\begin{align*}
\esp_{Z}\left[\frac{1}{d^{2}}\sum_{j=1}^{d}\alpha(W_{j},\underline{W}_{j}')\phi(Z,\underline{W}_{j}')(1-\phi(Z,\underline{W}_{j}'))\psi(Z,W_{j})^{2}\right] & \leq\esp_{Z}\left[\frac{1}{d^{2}}\sum_{j=1}^{d}\alpha(W_{j},\underline{W}_{j}')^{2}\psi(Z,W_{j})^{2}\right]\\
 & =\frac{1}{d^{2}}\sum_{j=1}^{d}\alpha(W_{j},\underline{W}_{j}')^{2}\esp_{Z}\psi(Z,W_{j})^{2}\\
 & \leq\frac{1}{d^{2}}\sum_{j=1}^{d}\alpha(W_{j},\underline{W}_{j}')^{2}L^{2}.
\end{align*}
Using that $\left(\alpha(W_{j},\underline{W}_{j}')^{2}\right)_{j}$ are an i.i.d.
sequences of integrable random variables, we obtain that 
\[
\lim_{d\to+\infty}\esp_{Z}\left[\frac{1}{d^{2}}\sum_{j=1}^{d}\alpha(W_{j},\underline{W}_{j}')\phi(Z,\underline{W}_{j}')(1-\phi(Z,\underline{W}_{j}'))\psi(Z,W_{j})^{2}\right]=0,
\]
almost-surely. 

\textbf{Convergence of bias term} $\alpha(W_{j},\underline{W}_{j}')\phi(Z,\underline{W}_{j}')\psi(Z,W_{j})$
is integrable because $|\alpha(W_{j},\underline{W}_{j}')\phi(Z,\underline{W}_{j}')|\leq|\alpha(W_{j},\underline{W}_{j}')|$,
and $\psi(z,W_{j})\in L^{2}(\nu)$. Then, using Kolmogorov's law and
mapping continuous theorem, we obtain 

\[
\lim_{d\to+\infty}\esp_{Z}\left[\left(Y-\frac{1}{d}\sum_{j=1}^{d}\alpha(W_{j},\underline{W}_{j}')\phi(Z,\underline{W}_{j}')\psi(Z,W_{j})\right)^{2}\right]=\esp_{Z}\left[\left(Y-\int\alpha(w,w')\phi(Z,w')\psi(Z,w)d\mu\otimes\nu(w,w')\right)^{2}\right].
\]

Thus we obtain, 
\[
\limsup_{d\to+\infty}R_{{\rm imp}}^{\star}(d)\leq\esp_{Z}\left[\left(Y-\int\alpha(w,w')\phi(Z,w')\psi(Z,w)d\mu\otimes\nu(w,w')\right)^{2}\right].
\]
Denoting by $\mathcal{G}$, the functions of the form 
\[
g(Z)=\int\alpha(w,w')\phi(Z,w')\psi(Z,w)d\mu\otimes\nu(w,w'),
\]
we obtain that 
\[
\limsup_{d\to+\infty}R_{{\rm imp}}^{\star}(d)\leq\inf_{g\in\mathcal{G}}\esp_{Z}\left[\left(Y-g(Z)\right)^{2}\right].
\]
Using Fubini theorem, functions of the form 
\begin{align*}
g(Z) & =\int\alpha(w)\psi(Z,w)d\nu(w)\int\beta(w')\phi(Z,w')d\mu(w')\\
 & =f(Z)h(Z),
\end{align*}
are include in $\mathcal{G}$. For the following, we denote by $\mathcal{H}$ the
set of functions, of the form 
\[
h(Z)=\int\beta(w')\phi(Z,w')d\mu(w'),
\]
with $\beta\in L^{2}(\mu)$. Thus we obtain, the following bound, 
\begin{equation}\label{eq:produitcaractarisation}
    \limsup_{d\to+\infty}R_{{\rm imp}}^{\star}(d)\leq\inf_{(f,h)\in\mathcal{F}_\nu\times\mathcal{G}}\esp_{Z}\left[\left(Y-f(Z)h(Z)\right)^{2}\right].
\end{equation}

\textbf{Proof for the case $\mathcal{Z}$ compact, $\mathcal{F}$ dense in
continuous function, and $f^{\star}$continuous }
First, let's show that the risk is continuous for the set of continuous prediction functions. Let $f,g$ two continuous function on $B_\infty(0,B)$ (ball for $\Vert\Vert_\infty$), then 
\begin{align*}
|R(f)-R(g)| & =\left|\esp\left[(f(Z)-f^{\star}(Z))^{2}-(g(Z)-f^{\star}(Z))^{2}\right]\right|\\
 & =\left|\esp\left[(f(Z)-g(Z))(g(Z)+f(Z)-f^{\star}(Z))\right]\right|\\
 & \leq\left\Vert f-g\right\Vert _{\infty}(\left\Vert f\right\Vert _{\infty}+\left\Vert g\right\Vert _{\infty}+\left\Vert f^{\star}\right\Vert _{\infty})\\
 & \leq\left\Vert f-g\right\Vert _{\infty}(2B+\left\Vert f^{\star}\right\Vert _{\infty}).
\end{align*}
That concludes that risk is continuous. 
Then, we
consider $\beta=1,$ thus $h(Z)=\int\beta(w')\phi(Z,w')d\mu(w')>0$
almost-surely because $\phi(Z,w')>0$ almost surely. Thus considering
in \eqref{eq:produitcaractarisation}, $f(Z)=f^{\star}(Z)/h(Z)$, $f$ is continuous, we conclude
using continuity of risk and $f\in\bar{\mathcal{F}}$.

\textbf{Proof for Gaussian RF} In this case $\phi(z,(w,w_0'))=\Phi(z^\top w'+w_0')$, consider $I=\int \Phi(w_0')d\mu $.
We consider $h_{\epsilon}(z)=\int_{\Vert w'\Vert\leq \epsilon}\phi(Z,(w',w_0')d\mu(w',w_0)/\int_{\Vert w'\Vert\leq \epsilon}\Phi(w_0')d\mu(w',w_0')$,
\begin{equation*}
    \vert I-h_{\epsilon}(Z)\vert\leq \frac{1}{\int_{\Vert w'\Vert\leq \epsilon}\Phi(w_0')d\mu(w')} \int_{\Vert w'\Vert\leq \epsilon}\vert\phi(0,(0,w_0'))-\phi(Z,(w',w_0))\vert d\mu(w',w_0').
\end{equation*}
Using that that $\Phi$ is $C$ Lipschitz,
\begin{equation*}
    \vert I-h_{\epsilon}(Z)\vert\leq C\epsilon\Vert Z\Vert.
\end{equation*}
We conclude using $h=h_\epsilon$ for a decreasing sequence of $\epsilon$ and $f=f^\star/I$ in \eqref{eq:produitcaractarisation}.

\end{document}